\documentclass[numbers,sort&compress]{elsarticle}

\usepackage[T1]{fontenc}
\usepackage[utf8]{inputenc}
\usepackage[american]{babel}

\usepackage{subfig}
\usepackage{tikz}
\usetikzlibrary
  {arrows,calc,through,backgrounds,matrix,positioning,decorations.pathmorphing}

\usepackage{amsmath}
\usepackage{amssymb}
\usepackage{amsthm}

\newtheorem{lemma}{Lemma}[section]
\newtheorem{proposition}[lemma]{Proposition}
\newtheorem{theorem}[lemma]{Theorem}
\newtheorem{definition}[lemma]{Definition}
\newtheorem{remark}[lemma]{Remark}

\newcommand{\Z}{\mathbb{Z}}
\newcommand{\R}{\mathbb{R}}
\newcommand{\C}{\mathbb{C}}
\newcommand{\N}{\mathbb{N}}
\newcommand{\e}{\varepsilon}
\newcommand{\eff}{\varepsilon^\text{eff}}

\newcommand{\curl}{\mathrm{curl}}

\renewcommand{\vec}[1]{\boldsymbol{#1}}
\newcommand{\vH}{\vec H}
\newcommand{\vE}{\vec E}
\newcommand{\vF}{\vec F}
\newcommand{\vJ}{\vec J}
\newcommand{\vJa}{\vec J_{a}}
\newcommand{\vPsi}{\vec \Psi}
\newcommand{\vPhi}{\vec \Phi}
\newcommand{\vJSd}{\vec J_{\Sigma^d}}
\newcommand{\vME}{\vec{\mathcal{E}}}
\newcommand{\vMH}{\vec{\mathcal{H}}}
\newcommand{\vn}{\vec \nu}

\newcommand{\vx}{\vec x}
\newcommand{\vy}{\vec y}

\newcommand{\ve}{\vec e}
\newcommand{\vu}{\vec u}
\newcommand{\vv}{\vec v}
\newcommand{\vf}{\vec f}
\newcommand{\vc}{\vec c}
\newcommand{\vk}{{\vec k}}
\newcommand{\vg}{{\vec g}}

\newcommand{\pa}{\partial}
\newcommand{\na}{\nabla}

\newcommand{\HT}{H_T(\curl;\Omega)}

\newcommand{\dx} {\,{\mathrm d}x}
\newcommand{\dox}{\,{\mathrm d}o_x}
\newcommand{\dy} {\,{\mathrm d}y}
\newcommand{\doy}{\,{\mathrm d}o_y}
\newcommand{\dt} {\,{\mathrm d}t}

\newcommand{\jso}[1]{\left[#1\right]_{\Sigma_0}}
\newcommand{\jsd}[1]{\left[#1\right]_{\Sigma^d}}

\renewcommand{\Re}{\mathrm {Re}\,}
\renewcommand{\Im}{\mathrm {Im}\,}

\def\wSigma{\widetilde\Sigma}
\def\wOmega{\widetilde\Omega}
\def\wE{\widetilde{\vE}}
\def\wF{\widetilde{\vF}}

\begin{document}

\title{%
  Homogenization of time-harmonic Maxwell's equations in nonhomogeneous %
  plasmonic structures}

\author[1]{Matthias Maier}
\ead{maier@math.tamu.edu}
\address[1]{%
  Department of Mathematics, Texas A\&M University,
  3368 TAMU, College Station, TX 77843, USA}

\author[2]{Dionisios Margetis}
\ead{dio@math.umd.edu}
\address[2]{%
  Department of Mathematics, and Institute for
  Physical Science and Technology, and Center for Scientific Computation
  and Mathematical Modeling, University of Maryland,
  College Park, Maryland 20742, USA}

\author[3]{Antoine Mellet}
\ead{mellet@math.umd.edu}
\address[3]{%
  Department of Mathematics, and Center for
  Scientific Computation and Mathematical Modeling, University of
  Maryland,
  College Park, Maryland 20742, USA}

\begin{abstract}
  We carry out the homogenization of time-harmonic Maxwell's equations in a
  periodic, layered structure made of two-dimensional (2D) metallic sheets
  immersed in a heterogeneous and in principle anisotropic dielectric
  medium. In this setting, the tangential magnetic field exhibits a
  \emph{jump} across each sheet. Our goal is the rigorous derivation of the
  effective dielectric permittivity of the system from the solution of a
  local cell problem via suitable averages. Each sheet has a fine-scale,
  inhomogeneous and possibly anisotropic surface conductivity that scales
  linearly with the microstructure scale, $d$. Starting with the weak
  formulation of the requisite boundary value problem, we prove the
  convergence of its solution to a homogenization limit as $d$
  approaches zero. The effective permittivity and cell problem express a
  bulk average from the host dielectric and a \emph{surface average}
  germane to the 2D material (metallic layer). We discuss implications of
  this analysis in the modeling of plasmonic crystals.
\end{abstract}


\maketitle


\section{Introduction}
\label{sec:introduction}

Recent advances in the design and synthesis of thin materials have
challenged traditional notions of optics such as the diffraction limit. The
emerging class of metamaterials enable the control of the path and
dispersion of light, which may in turn result in unusual optical phenomena
that include no refraction (``epsilon near zero'' effect) and negative
refraction~\cite{maier2019,maier2018,mattheakis2016,LiKita2015,moitra2013,Silveirinha2006,Wang2012}.
In fact, the optical conductivity of certain two-dimensional (2D) materials
in the infrared spectrum permits the excitation of short-scale
electromagnetic surface waves, called surface plasmon-polaritons, in the
electron plasma under the appropriate polarization of the incident
field~\cite{Lowetal17,Bludov13,grigorenko2012,Pitarke07}. This type of wave
is tightly confined near the 2D material. The existence of this wave has
inspired the design of layered plasmonic structures that exhibit
unconventional optical properties via the tuning of frequency or
geometry~\cite{grigorenko2012,mattheakis2016,maier2018,maier2019}. This physical
prospect motivates our present work.

In a layered structure of 2D metallic sheets, the surface
plasmon-polaritons excited in the electron plasma of the layers may
constructively interfere in the dielectric host. This wave coupling can be
enhanced for small enough interlayer spacing at the microscale; and can
give rise to a slowly varying wave that propagates through the structure at
the macroscale. By a suitable adjustment of the operating frequency or
interlayer spacing, this wave may experience no phase
delay~\cite{mattheakis2016,maier2018}. Mathematically, it is tempting to
view this possibility as an outcome of homogenization, expecting that there
is an \emph{effective description} of wave propagation as the spacing
approaches zero. The phase of the optical conductivity of each sheet plays
a key role.

In this paper, we rigorously carry out the homogenization of a boundary
value problem for the time harmonic Maxwell equations in a periodic,
layered structure. The geometry consists of 2D plasmonic sheets in a
heterogeneous dielectric medium. The surface conductivity, $\sigma^d$, of
each sheet varies spatially with the microstructure scale $d$ and may be
anisotropic. The dielectric permittivity, $\e^d$, of the host medium has an
analogous, $d$-periodic microstructure in the ambient space and can be
anisotropic. Our main result is the rigorous extraction of an
\emph{effective} dielectric permittivity and the related cell problem as
$d\to 0$. Specifically, we complete the following main tasks.

\begin{itemize}
  \item
    We develop the weak formulation for the associated boundary value
    problem of Maxwell's equations for the  electromagnetic field $(\vE^d,
    \vH^d)$ in some generality. The tangential vector component, $\vH^d_T$,
    of the magnetic field obeys a \emph{jump condition} on each sheet; and
    the jump is proportional to $\sigma^d \vE^d_T$, the current induced on
    the sheet.  We make the assumption that $\sigma^d$ scales linearly with
    $d$, which is consistent with the experimentally observed fine-scale
    surface plasmon-polaritons.
  \item
    We address the simplified case with planar sheets, and scalar $\e^d$
    and $\sigma^d$  first. In this vein, we prove a theorem
    (Theorem~\ref{thm:existence}) asserting that for fixed $d$, the weak
    formulation admits a unique solution in an appropriate function space.
  \item
    We then show that the electromagnetic field ($\vE^d, \vH^d$) converges
    weakly in $L^2$ to the solution  ($\vME,\vMH$) of the homogenized
    problem (Theorem~\ref{thm:hom}). The homogenization limit reveals
    the effective permittivity, $\eff$, via a suitable average and the
    solution of a local cell problem; cf.~\eqref{eq:effectiveperm}. To
    obtain these results, we establish requisite a priori estimates in the
    context of two-scale convergence. For an overview of important results
    related to two-scale convergence, see
    \ref{app:two-scale}.
  \item
    We discuss the relevance of our model and analysis to the application
    area of plasmonics, especially the design of plasmonic crystals that
    exhibit no refraction (epsilon-near-zero effect).
  \item
    We point out extensions of our analysis to more general settings. In
    particular, our analysis can treat tensorial parameters $\varepsilon^d$
    and $\sigma^d$, and non-planar sheets (see
    \ref{app:generalization}).
\end{itemize}

In our analysis, for the sake of mathematical convenience we assume that
the bulk material surrounding the metallic sheets is slightly lossy. This
assumption, which is not uncommon in electromagnetics~\cite{Muller69},
amounts to the addition of a small, positive imaginary part to the
dielectric permittivity $\e^d$ (under an $e^{-i \omega t}$ time
dependence). Consequently, we conveniently obtain the desired a priori
estimates for $(\vE^d,\vH^d)$.

There is extensive literature in the theory of periodic homogenization that
is akin to our approach; see, e.g.,~\cite{allaire, nguetseng, papanicolaou,
cioranescu, sanchezpalencia, wellander2001, wellander2002, wellander2003,
amirat2011, amirat2017, artola}. Notably, the idea underlying the two-scale
asymptotic analysis for ($\vE^d, \vH^d$) can be found
in~\cite{papanicolaou,pavliotis}; and our proof of homogenization relies on
the known notion of two-scale convergence~\cite{allaire,nguetseng}. In the
setting of time-harmonic Maxwell's equations, our analysis brings forth the
feature of averaging on hypersurfaces (metallic sheets) across which the
magnetic field undergoes a jump involving the surface conductivity,
$\sigma^d$. A similar jump condition is considered in~\cite{amirat2017},
albeit in a different geometric setting which is motivated by geophysical
applications: In~\cite{amirat2017} the jump condition accounts for
interfacial currents that are present along the {\em closed surfaces} that
separate two distinct phases of a composite material (with periodic
structure). In our setting, on the other hand, surface currents on {\em
large 2D sheets} are discussed. This suggests a different approach to
handle the contribution of these currents in our proof than the one
employed in~\cite{amirat2017}. We refer to Sections~\ref{subsec:past}
and~\ref{sec:conclusion} for further discussion and a comparison of the two
problems and respective approaches.

We focus on the rigorous analysis of the periodic homogenization for
plasmonic layered structures. Hence, numerical computations tailored to
applications lie beyond our present scope, and will be the subject of
future work. We assume that the reader is familiar with the fundamentals of
classical electromagnetic wave theory; for extensive treatments of this
subject, see, e.g.,~\cite{Muller69,Schwartz72}. The $e^{-i\omega t}$ time
dependence is employed throughout.

\subsection{Problem formulation}
\label{subsec:formulation}

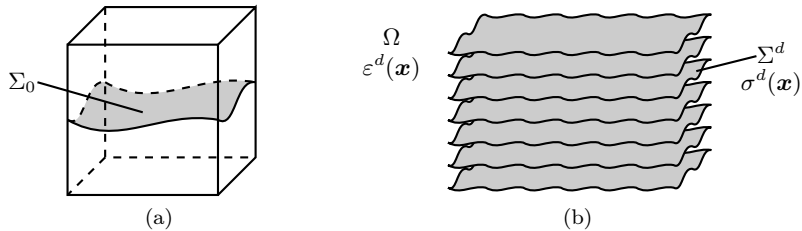
\begin{figure}[tb]
  \centering
    \subfloat[]{
      \hspace{-3em}
      \begin{tikzpicture}[scale=1.00]
        \path [thick, draw]
          (-1, -1) -- (1, -1) -- (1, 1) -- (-1, 1) -- cycle;
        \path [thick, draw] (1,-1) -- (1.5, -0.5) -- (1.5, 1.5) -- (1, 1);
        \path [thick, draw] (-1,1) -- (-0.5, 1.5) -- (1.5, 1.5);
        \path [thick, draw, dashed]
          (-0.5, 1.5) -- (-0.5, -0.5) -- (1.5, -0.5);
        \path [thick, draw, dashed] (-0.5, -0.5) -- (-1.0, -1.0);
        \path [fill=black, fill opacity=0.2]
          (-1, 0) to[out=-30,in=170] (1, 0) to[out=-30,in=170] (1.5, 0.5)
          to[out=170,in=-30] (-0.5, 0.5) to[out=170,in=-30] (-1.0, 0.0);

        \path [thick, draw] (-1, 0) to[out=-30,in=170] (1, 0);
        \path [thick, draw] (1, 0) to[out=-30,in=170] (1.5, 0.5);
        \path [thick, draw, dashed] (-0.5, 0.5) to[out=-30,in=170] (1.5, 0.5);
        \path [thick, draw, dashed] (-1.0, 0.0) to[out=-30,in=170] (-0.5, 0.5);
        \path [thick, draw] (-1.4, 0.5) -- (0.0, 0.1);
        \node at (-1.6, 0.5) {\small $\Sigma_0$};
      \end{tikzpicture}}
    \hspace{3em}
    \subfloat[]{
      \begin{tikzpicture}[scale=1.00]
        \foreach \x in {0,...,6} {
          \path [fill=white]
            (-1.5, -1+0.3*\x) to[out=-30,in=170] (-1.0, -1+0.3*\x)
                               to[out=-30,in=170] (-0.5, -1+0.3*\x)
                               to[out=-30,in=170]  (0.0, -1+0.3*\x)
                               to[out=-30,in=170]  (0.5, -1+0.3*\x)
                               to[out=-30,in=170]  (1.0, -1+0.3*\x)
                               to[out=-30,in=170]  (1.5, -1+0.3*\x)
                               to[out=-30,in=170] (1.75, -0.75+0.3*\x)
                               to[out=-30,in=170] (2.00, -0.5+0.3*\x)
                               to[out=170,in=-30] (1.50, -0.5+0.3*\x)
                               to[out=170,in=-30] (1.00, -0.5+0.3*\x)
                               to[out=170,in=-30] (0.50, -0.5+0.3*\x)
                               to[out=170,in=-30] (0.00, -0.5+0.3*\x)
                               to[out=170,in=-30] (-0.5, -0.5+0.3*\x)
                               to[out=170,in=-30] (-1.0, -0.5+0.3*\x)
                               to[out=170,in=-30] (-1.25, -0.75+0.3*\x)
                               to[out=170,in=-30] cycle;
          \path [thick, draw, fill=black, fill opacity=0.2]
            (-1.5, -1+0.3*\x) to[out=-30,in=170] (-1.0, -1+0.3*\x)
                               to[out=-30,in=170] (-0.5, -1+0.3*\x)
                               to[out=-30,in=170]  (0.0, -1+0.3*\x)
                               to[out=-30,in=170]  (0.5, -1+0.3*\x)
                               to[out=-30,in=170]  (1.0, -1+0.3*\x)
                               to[out=-30,in=170]  (1.5, -1+0.3*\x)
                               to[out=-30,in=170] (1.75, -0.75+0.3*\x)
                               to[out=-30,in=170] (2.00, -0.5+0.3*\x)
                               to[out=170,in=-30] (1.50, -0.5+0.3*\x)
                               to[out=170,in=-30] (1.00, -0.5+0.3*\x)
                               to[out=170,in=-30] (0.50, -0.5+0.3*\x)
                               to[out=170,in=-30] (0.00, -0.5+0.3*\x)
                               to[out=170,in=-30] (-0.5, -0.5+0.3*\x)
                               to[out=170,in=-30] (-1.0, -0.5+0.3*\x)
                               to[out=170,in=-30] (-1.25, -0.75+0.3*\x)
                               to[out=170,in=-30] cycle;
        }
        \node at (-2.25, +1.0) {\small $\Omega$};
        \node at (-2.25, +0.6) {\small $\varepsilon^d(\vx)$};
        \node at (2.8, +0.8) {\small $\Sigma^d$};
        \path [thick, draw] (2.6, 0.75) -- (1.8, 0.55);
        \node at (2.8,  0.4) {\small $\sigma^d(\vx)$};
      \end{tikzpicture}}

  \caption{Geometry of the problem.
    (a) The unit cell, $Y=[0,1]^3$, with hypersurface $\Sigma_0$, (b) a
    layered structure consisting of parallel, conducting sheets $\Sigma^d$
    equipped with a spatially dependent surface conductivity
    $\sigma^d(\vx)$. We assume that the layered structure is immersed in a
    (unbounded) medium with a spatially dependent permittivity
    $\e^d(\vx)$.}
  \label{fig:layered}
\end{figure}
Our goal with this work is to extract \emph{effective} material parameters
in time-harmonic Maxwell's equations for layered systems of stacked,
metallic sheets immersed in a non-homogeneous medium.

For ease of discussion, we will introduce the homogenization problem in
this section for the (infinite) domain $\mathbb{R}^3$, assuming suitable
boundary conditions. The actual proof assumes a more restrictive
``reference configuration'' which is described in
Section~\ref{sec:rigorous_result}.
The geometry is shown in Figure~\ref{fig:layered}. The (complex) surface
conductivity $\sigma^d$ of every sheet, which  is in principle frequency
($\omega$-) dependent, has real and imaginary parts that can be tuned to
allow for the propagation of surface plasmon-polaritons on the isolated
sheet~\cite{maier2017,maier2019}.
\medskip

The \emph{scaling parameter} $d$, $d\ll 1$, describes the fine scale of the
problem and in particular the distance separating the conducting sheets: Let
$Y=[0,1]^3$ denote the unit cell and let $\Sigma_0$ denote a smooth
hypersurface in $Y$ (with smooth, periodic continuation); see
Figure~\ref{fig:layered}a. We then define a union of stacked, disconnected
hypersurfaces by
\begin{align*}
  \Sigma^d = \bigcup_{\vec z\in\Z^n} d\,(\vec z+\Sigma_0);
\end{align*}
see Figure~\ref{fig:layered}b. For all $\vx\in\Sigma^d$, the surface
conductivity $\sigma^d(\vx)$ is a tensor acting on the tangent space
$T_{\vx} \Sigma^d$. We assume that $\sigma^d(\vx)$ exhibits both fine-scale
(periodic) and large-scale variations in space. We also assume that
$\sigma^d\sim d$ so that the total conductivity remains finite when
$d\ll1$. (This scaling is also consistent with the fact that $|\sigma^d|$
must be small enough for the appearance of a fine-scale surface
plasmon-polariton on an isolated sheet~\cite{maier2017,maier2019}). We thus
write
\begin{align*}
  \sigma^d(\vx) = d\,\sigma\big(\vx,\vx/d\big),
\end{align*}
where $\sigma(\vx,\vy)$ is a tensor acting on the tangent space $T_{\vy}
\Sigma_0$ for all $(\vx,\vy)\in \R^3\times \Sigma_0$ and is independent of
$d$ and periodic with respect to $\vy$. Similarly, the permittivity of the
ambient medium $\e^d(\vx)$, defined for $\vx\in \R^3\setminus \Sigma^d$, is
assumed to be given by $$  \e^d(\vx) = \e\big(\vx,\vx/d\big)$$ for some
tensor $\e(\vx,\vy)$ independent of $d$. The quantities $\e(\vx,\vy)$ and
$\sigma(\vx,\vy)$ will be henceforth referred to as the rescaled
permittivity and surface conductivity, respectively. Additional conditions
on the domain, geometry and material parameters are provided below.

We now consider time-harmonic Maxwell's equations written in the form
\begin{align}
  \begin{cases}
    \begin{aligned}
      \nabla\times\vE^d &= i\omega\mu\vH^d,
      \\[0.3em]
      \nabla\times\vH^d &= -i\omega\e^d\vE^d+\vJa.
    \end{aligned}
  \end{cases}
  \label{eq:maxwell}
\end{align}
Here, $\vE^d$ and $\vH^d$ denote the electric and magnetic field,
respectively; and $\vJa$ is the (externally applied) source current
density. The parameter $\mu$ denotes the magnetic permeability of the
ambient space; e.g., $\mu=\mu_0$, a scalar constant, for the vacuum.
\medskip

In order to write the boundary conditions across $\Sigma^d$, we introduce
the normal unit vector field $\vn(\vx)$ on $\Sigma^d$, and let
$\jsd{\,.\,}$ denote the jump over $\Sigma^d$:
\begin{align*}
  \jsd{\vec F}(\vx)
  \,:=\,
  \lim_{\alpha\searrow0}
  \Big(\vec F(\vx+\alpha\vn) - \vec F(\vx-\alpha\vn)\Big)\qquad
  \vx\in\Sigma^d.
\end{align*}
We also denote by $\vF_T$ the tangential component of any vector field $\vF$,
viz., $\vF_T=(\vn\times\vF)\times\vn$. The current density induced on the
sheets because of the effect of $\sigma^d$ is
$\vJSd=\vec\delta_{\Sigma^d}\sigma^d\vE^d_T$. Hence, the boundary
conditions for the tangential components of the electromagnetic field
across $\Sigma^d$ read~\cite{maier2017,maier2019}
\begin{align}
  \begin{cases}
    \begin{aligned}
      \jsd{\vn\times\vE^d} &= 0,
      \\[0.3em]
      \jsd{\vn\times\vH^d} &= \sigma^d\vE^d_T.
    \end{aligned}
  \end{cases}
  \label{eq:jump}
\end{align}

\subsection{Main result}
\label{subsec:mainresult}

We will show that as $d\to 0$, the electric and magnetic fields, $\vE^d$
and $\vH^d$, converge to the solutions, $\vME$ and $\vMH$, of the
homogenized system
\begin{align}
  \begin{cases}
    \begin{aligned}
      \nabla\times\vME &= i\omega\mu\vMH,
      \\[0.3em]
      \nabla\times\vMH &= -i\omega\eff\vME+\vJa,
    \end{aligned}
  \end{cases}
  \label{eq:maxwellhomogenized}
\end{align}
where the \emph{effective permittivity} $\eff=\eff(\vx)$ is given as an
appropriate average involving the $d$-independent (rescaled) permittivity
and conductivity $\e$ and $\sigma$:
\begin{multline}
  \eff(\vx) :=
  \int_Y\e(\vx,\vy)(I_3+\nabla_y\vec\chi(\vx,\vy))\dy
  \\
  -\frac1{i\omega}
  \int_{\Sigma_0}\big(\sigma(\vx,\vy)P_T
  (I_3+\nabla_y\vec\chi(\vx,\vy))\big)\doy.
  \label{eq:effectiveperm}
\end{multline}
In the above, $I_3$ denotes the identity matrix in $\R^3$,  $P_T$ is the
projection matrix onto the tangent set of $\Sigma_0$ and  $\nabla_y
\vec\chi(\vx,\vy)$ denotes the matrix $[\pa_{y_i} \chi_j(\vx,\vy)]$ ($i,
j=1,2,3$). The corrector $y\mapsto \vec\chi(\vx,\vy)$ solves the following
\emph{cell problem} (for all $\vx$ in some bounded open set $\Omega\subset
\mathbb{R}^3$):
\begin{align}\label{eq:cell-intro}
  \begin{cases}
    \begin{aligned}
      \nabla_y\cdot\Big(\e(\vx,\vy)\big(I_3+\nabla_y\vec\chi(\vx,\vy)\big)\Big)
      &= 0,& &\mbox{in } Y
      \\[0.3em]
      \jso{\vn\cdot\Big(\e(\vx,\vy)\big(I_3+\nabla_y\vec\chi(\vx,\vy)\big)}
      &=
      \\[0.3em]
      \frac{1}{i\omega}\nabla_y\cdot&\Big(\sigma(\vx,\vy) P_T
      \big(I_3+\nabla_y\vec\chi(\vx,\vy)\big)\Big)& &\mbox{on } \Sigma_0.
    \end{aligned}
  \end{cases}
\end{align}
As mentioned in the introduction, this cell problem is similar to
equation~(34) in~\cite{amirat2017} for the case (scaling regime) of
``strong interface layer''.
Note that even when $\e$ and $\sigma$ are scalars, the effective
permittivity $\eff$ is a $3\times 3$ matrix as typically expected for the
case of a bulk material, that is, if $\sigma\equiv
0$~\cite{kristensson}.
%
%
Under the assumptions of Theorem \ref{thm:hom} (see \eqref{eq:hyp3}), we
have
\begin{align*}
  \Im \big( \eff(\vx) \vec \xi\cdot \overline {\vec \xi} \big) \geq
  c\left(1+\frac1 \omega\right) |\vec\xi|^2,
\end{align*}
which ensures the well-posedness of homogenized system
\eqref{eq:maxwellhomogenized}.

We alert the reader that the proofs in this paper are only developed in the
scalar case, i.e., when the material parameters $\varepsilon^d$ and
$\sigma^d$ are scalars, for the sake of simplicity. On the other hand, the
above statements are written more generally, for tensorial $\varepsilon^d$
and $\sigma^d$. Our analysis can be extended to the tensor case without
difficulties.

\subsection{Novelty and application}
\label{subsec:novel}

The insertion of an array of metallic sheets, each of which can sustain
surface plasmon-polaritons, into dielectric hosts with small enough
interlayer spacing has significant physical
appeal~\cite{moitra2013,mattheakis2016,maier2018}. From an analysis
perspective, this type of structure motivates the homogenization procedure
of this paper, and leads to an intriguing homogenization result;
cf.~Section~\ref{subsec:mainresult}.

Foremost, effective permittivity \eqref{eq:effectiveperm} is now the
\emph{combination} of two averages, namely, one average stemming from the
ambient-medium permittivity tensor $\e(\vx,\vy)$, and another from the
surface conductivity $\sigma(\vx,\vy)$ of each metallic sheet. To our
knowledge, this combination of two effective parameters, one coming from a
bulk property and another expressing the property of a hypersurface, has
not occurred in most of the previous homogenization results; see,
e.g.,~\cite{allaire, wellander2003}. An exception is the homogenization of
Maxwell's equations in a two-phase composite material carried out
in~\cite{amirat2017}, which we mentioned above; see also the discussion in
Sections~\ref{subsec:past}, and~\ref{sec:conclusion}.

For applications in plasmonics, the case with a surface conductivity,
$\sigma^d$, that has a dominant imaginary part, viz.,
$\text{Im}\,\sigma^d\gg\text{Re}\,\sigma^d>0$ in the case of a scalar
$\sigma^d$, has attracted particular attention. Such a surface conductivity
can be created with novel 2D materials, for example graphene
\cite{grigorenko2012}. By carefully tuning the frequency, geometry, or
the surface conductivity, $\sigma$, via doping of the 2D material, one may
obtain $\eff$ with eigenvalues that have vanishing, or negative real part.
The homogenized system described by \eqref{eq:effectiveperm} can thus be
viewed as a \emph{metamaterial} exhibiting highly unusual optical phenomena
such as the epsilon-near-zero effect or negative refraction
\cite{moitra2013,mattheakis2016,maier2018}. This implication and the
connection of our homogenization  result to existing predictions of
epsilon-near-zero behavior are discussed in
Section~\ref{sec:conclusion}.

\subsection{On past works}
\label{subsec:past}

Our analysis relies on firm concepts of homogenization
theory~(\cite{allaire, nguetseng, papanicolaou, cioranescu,
sanchezpalencia}), which we employ in the setting of electromagnetic wave
propagation in the presence of 2D plasmonic materials. Over the past
decade, numerous studies have been conducted on
related applications, especially because of the prospect of fabricating
metamaterials with unusual properties in nanophotonics. These properties
come from combining and averaging out suitable microstructures. For recent
reviews from an applied physics perspective, we refer the reader to
\cite{Caldwell2016,Jahani2016,Zheludev2016}.

From the viewpoint of analysis, we should highlight a number of
homogenization results~\cite{wellander2001, wellander2002, wellander2003,
amirat2011, amirat2017, artola} that are relevant to our problem
formulation, as well as the germane notion of two-scale
convergence~\cite{allaire, nguetseng} which underlies our approach. In
particular, the homogenization results obtained in~\cite{wellander2001,
wellander2002, wellander2003, amirat2011,artola}  coincide with
\eqref{eq:effectiveperm} in the special case with a vanishing surface
conductivity, $\sigma(\vx, \vy)=0$.

As mentioned above, our work is related to that in~\cite{amirat2017} which
also analyzes the effect of surface currents in the homogenization of the
time harmonic Maxwell equations. In~\cite{amirat2017}, however, the authors
focus on a different geometric and physical setting, which involves
two-phase materials with certain inclusions. Various regimes, which depend
on the strength of the interfacial currents, the wavelength and the skin
depth, are studied in~\cite{amirat2017}. Both the mathematical formulation
and the homogenized equations that we derive in the present paper are
closely related to the case referred to as the {\it strong interface layer}
in~\cite{amirat2017} (see Theorem 2 in \cite{amirat2017}).

However, there are key differences in the geometric setting of the two
problems. In~\cite{amirat2017} the authors consider a material (e.g., clay)
containing periodically distributed rock inclusions; thus, the
hypersurfaces are boundaries of small disconnected sets rather than the
large 2D sheets studied here. Most notably, in~\cite{amirat2017}, a key
tool in \cite{amirat2017} is the generalization of the notion of two-scale
convergence to functions defined on periodic surfaces \cite{Neuss-Radu,
Allaire95}, which exploits the geometric setting of small inclusions in a
crucial way. The proof that we develop in the present paper, however, is
directly motivated by our geometry consisting of large 2D interfaces and
does not rely on this notion of two-scale convergence on surfaces but
instead recovers directly the convergence of the currents to the
appropriate term in the sense of distribution; cf. Proposition
\ref{prop:limsum}. We point out that in principle our results can be
extended to the geometry of closed inclusions considerd in
\cite{amirat2017}, and vice versa. In that sense we offer an independent
proof of the homogenization result.

We should also mention a number of related results for periodic media that
are obtained by use of the Bloch wave theory; see, e.g.,
\cite{sjoeberg2005} where the authors conclude that only a few Bloch waves
effectively contribute to the macroscopic field. In a similar physical
context, the application of homogenization to finite photonic crystals is
described in \cite{guenneau}; and its connection to certain numerical
multiscale methods for Maxwell's equations in composite materials  is
elaborated in~\cite{cao, henning}. In particular, in~\cite{guenneau} the
authors formally apply a two-scale  asymptotic expansion to derive
effective bulk parameters that take into account the crystal boundary. To
these works we add the homogenization of Maxwell's equations in the
presence of rough boundaries and interfaces pioneered in~\cite{Nevard97}.

\subsection{Outline}
\label{subsec:outline}

The remainder of the paper is organized as follows. In
Section~\ref{sec:rigorous_result}, we introduce the weak formulation of the
problem, along with two key theorems (Theorems~\ref{thm:existence}
and~\ref{thm:hom}) which permeate our analysis. Section~\ref{sec:existence}
focuses on the proof of one of these theorems
(Theorem~\ref{thm:existence}), namely, the existence of a weak solution to
Maxwell's equations for a finite microstructure scale, $d>0$. In
Section~\ref{sec:convergence}, we prove the second key theorem
(Theorem~\ref{thm:hom}) which establishes the convergence of the weak
solution for $d>0$ to the homogenization limit as $d\to 0$ and
recovers the cell problem. Finally, in Section~\ref{sec:conclusion} we conclude
our work with a discussion of its relevance to the design of plasmonic
crystals with unusual optical properties. In Section~\ref{sec:conclusion},
an outline of related open problems is given as well.
\ref{app:two-scale} provides an overview of important results in two-scale
convergence. \ref{app:generalization} contains details of the extension of
our analysis to more general hypersurfaces for the metallic sheets. At the
risk of redundancy, we repeat that our analysis in this paper focuses on
the scalar case. The extension of our proofs to tensorial parameters does
not present any difficulties, and is not pursued here.


\section{Homogenization of layered structures}
\label{sec:rigorous_result}

In this section, we introduce a weak formulation of \eqref{eq:maxwell} and
\eqref{eq:jump}, and state the main theorems of this paper. At this stage,
we will analyze a special choice of geometry (``reference configuration''),
in which the hypersurfaces are flat. (We outline an extension of our proof
to curved hypersurfaces in \ref{app:generalization}). Specifically, let
$\Omega$ be a bounded open set in $\R^3$ of the form
$\Omega=\Sigma_0\times\Gamma$, where $\Sigma_0$ is a (bounded) subset of
$\R^2$ and $\Gamma=(-L,L)$. The layered structure is then described by
\begin{align*}
  \Sigma^d =  \bigcup_{k\in \Gamma^d} \Sigma_0\times \{kd\},
\end{align*}
where $\Gamma^d = \{k\in \Z\,;\, kd\in (-L,L-d)\}$.
\begin{remark}
  We point out that our homogenization result holds for a larger class of
  geometries. In fact, all theorems and proofs that are presented in this
  section can be readily extended to cases of periodic structures that are
  diffeomorphic to the above reference configuration (as depicted in
  Figure~\ref{fig:layered}). We outline an extension of our proof to such
  curved hypersurfaces in \ref{app:generalization}.
\end{remark}

We will use the notation $\vx=(\vx',x_3)\in\Sigma\times\Gamma$ when
necessary. For later convenience, we also define $\tilde\Gamma^d =
\bigcup_{k\in \Gamma^d}\big[kd,(k+1)d\big]$. Hence, we have
\begin{equation}
  \label{eq:Gamma}
  (-L+d,L-d)\subset \tilde\Gamma^d\subset (-L,L).
\end{equation}

\subsection{Weak formulation}
\label{subsec:weak-form}

Formally, our problem is equivalent to the system
\begin{align}
  \begin{cases}
    \begin{aligned}
      \nabla\times\vE^d &= i\omega\mu\vH^d,
      \\[0.5em]
      \nabla\times\vH^d &= -i\omega\e^d\vE^d+\vJ^d,
    \end{aligned}
  \end{cases}
  \qquad\text{in }\Omega
\label{eq:strong}
\end{align}
where
\begin{equation}
  \label{eq:Jd}
  \vJ^d =  (\sigma^d \vE^d_T)\,\delta_{\Sigma^d} + \vJa.
\end{equation}
Note that \eqref{eq:strong} implies the following Helmholtz-type equation
for the electric field
\begin{equation}
  \label{eq:helm}
  \frac{1}{\mu} \nabla\times(\nabla\times\vE^d) = \omega^2 \e^d \vE^d +
  i\omega\vJ^d,
\end{equation}
along with the relation $\vH^d = \frac{1}{i\omega\mu}
\nabla\times\vE^d$. On the boundary of $\Omega$, we supplement this
equation with the following impedance boundary condition \cite{Monk03}:
\begin{align}
  \frac{1}{\mu}(\nabla\times
  \vE^d)\times\nu = i\omega\lambda \vE^d_T \qquad \mbox{ on }\partial\Omega.
\label{eq:BC}
\end{align}
Recall that $\nu$ denotes the normal unit vector. Boundary condition
\eqref{eq:BC} is often used for scattering configurations: For the
particular choice $\lambda=\sqrt{\mu^{-1}\e}$, formula \eqref{eq:BC}
recovers a first-order absorbing boundary condition \cite{Monk03}.

By multiplying~\eqref{eq:helm} by the conjugate of a smooth test
function $\vPsi$, we obtain the following formulation for Maxwell's
equations in domain $\Omega$:
\begin{align*}
  & \int_{\Omega  } \frac 1 {\mu} (\nabla\times \vE^d )\cdot
  (\nabla\times \bar \vPsi) \, \dx +\int_{\partial \Omega}  \frac 1 {\mu}
  \nu\times (\nabla\times \vE^d) \cdot \bar \vPsi_T\, \dx
  \\
  & \qquad \qquad = \int_{\Omega } \omega^2 \e^d \vE^d \cdot \bar \vPsi\,
  \dx +  i\omega\int_\Omega \vJ^d\cdot \vPsi \, \dx.
\end{align*}
By substituting expression~\eqref{eq:Jd} for current $\vJ^d$ and using 
boundary condition~\eqref{eq:BC} on $\partial\Omega$, we thus obtain the
following weak formulation:
\begin{multline}
  \label{eq:weak}
  \int_{\Omega  } \frac 1 {\mu} (\nabla\times \vE^d)\cdot (\nabla\times
  \overline \vPsi) \, \dx -\int_{\partial \Omega}  i\omega\lambda  \vE^d_T
  \cdot  \overline\vPsi_T\, \dox
  \\
  = \int_{\Omega } \omega^2 \e^d  \vE^d \cdot  \overline \vPsi\dx +
  \int_{\Sigma^d} i\omega \sigma^d \vE^d_T\cdot \overline \vPsi_T\, \dox +
  \int_{\Omega } i\omega \vJa \cdot  \overline\vPsi\, \dx.
\end{multline}
\begin{remark}
  In order to obtain the desired a priori estimates below, we will need to
  assume that the bulk material (ambient medium) is dissipative. This
  property can be ensured by addition of the current density $\sigma_0
  \vE^d$ in the bulk material, where $\sigma_0>0$. Alternatively, this
  current ensues by addition of the term $i\frac{\sigma_0}{\omega}$ to the
  permittivity $\e^d$. This is the reason why $\e^d$ will be complex
  valued, with strictly positive imaginary part, in this section (see also
  conditions \eqref{eq:hyp1}-\eqref{eq:hyp1'} and Remark~\ref{rem:matrix}).
\end{remark}

\subsection{Function spaces}
\label{subsec:function-sp}

The space $H(\curl;\Omega)$ denotes the set of complex valued vector
functions $\vu\in L^2(\Omega;\C^3)$ such that $\nabla\times \vu \in
L^2(\Omega;\C^3)$. Given a hypersurface $\Lambda \subset \overline \Omega$,
the  trace $\nu\times \vu|_\Lambda$ (where $\nu$ is the unit normal vector
to $\Lambda$) is well defined for functions in $H(\curl;\Omega)$ and it
belongs to $H^{-\frac 1 2} (\mathrm{div}_\Lambda,\Lambda)$. We denote by $
\vu_T = (\nu\times \vu)\times\nu$ the tangential component of $\vu$. In
view of \eqref{eq:weak}, the function space that is natural for our problem
is
\begin{align*}
  X^d =\big\{ \vu \in H(\curl;\Omega) \,;\, \vu_T\in L^2(\partial\Omega;\C^3),
  \; \vu_T \in L^2(\Sigma^d;\C^3)\big\},
\end{align*}
equipped with the norm
\begin{align*}
  \|\vu\|_{X^d}^2  = \| \vu\|^2_{L^2(\Omega)}+  \|\nabla\times
  \vu\|^2_{L^2(\Omega)} + \| \vu_T\|^2_{L^2(\partial \Omega)} + \|
  \vu_T\|^2_{L^2(\Sigma^d)},
\end{align*}
where
\begin{align}\label{eq:Sigmanorm}
  \|\vu_T\|^2_{L^2(\Sigma^d)}
  = d \int_{\Sigma^d }|\vu_T|^2\, \dox = \sum_{k\in \Gamma^d}
  d\int_{\Sigma} |\vu_T(x',kd)|^2\, \dx'.
\end{align}
Note the presence of the factor $d$ in this last norm. This is a natural
definition since with this scaling the $L^2(\Sigma^d)$ norm is a Riemann
sum approximation of the $L^2(\Omega)$ norm when $d\ll 1$. With these
prerequisites at hand we can introduce a precise problem formulation:
Equation \eqref{eq:weak} is equivalent to
\begin{equation}
  \label{eq:weaka}
  \vE^d \in X^d,
  \qquad
  a^d(\vE^d,\vPsi) = i\omega \int_{\Omega } \vJa \cdot  \overline\vPsi\,
  \dx  \qquad \forall\, \vPsi \in X^d,
\end{equation}
where the sesquilinear form $a^d:X^d\times X^d\to \C$ is defined by
\begin{align*}
  a^d(\vu,\vv) &=
  \int_{\Omega  } \frac 1 {\mu} (\nabla\times \vu)\cdot (\nabla\times
  \overline \vv) \, \dx - \int_{\Omega } \omega^2 \e^d  \vu \cdot
  \overline \vv\, \dx\\ &\qquad\qquad -\int_{\Sigma^d} i\omega \sigma^d
  \vu_T\cdot \overline \vv_T\, \dox -\int_{\partial \Omega} i\omega\lambda
  \vu_T \cdot  \overline\vv_T\, \dox.
\end{align*}
We will show that \eqref{eq:weaka} uniquely determines the electric field
$\vE^d$ (see~Theorem \ref{thm:existence}). The corresponding magnetic field
is then given by
\begin{equation}
  \label{eq:mag}
  \vH^d (\vx)= \frac{1}{i\omega\mu} \nabla\times \vE^d(\vx).
\end{equation}

\subsection{Main theorems}
\label{subsec:theorems}
We are now ready to state our main results. Throughout the paper, we will
make the following assumption on the properties of the material:
\medskip

\noindent
{\bf Assumption:} {\it We assume (for simplicity) that $\mu$ and $\lambda$
are positive, real, scalar constants and that $\e^d$ and $\sigma^d$ are
measurable complex valued functions  satisfying
\begin{equation}
  \label{eq:hyp1}
  0 < c \leq  \Im \e^d (\vx ) \leq C
  \quad
  \text{and}
  \quad
  |\Re \e^d(\vx )|\leq C \qquad \forall \vx\in \Omega
\end{equation}
and
\begin{equation}
  \label{eq:hyp1'}
  0 < c \leq  \frac{1}{d}\,\Re \sigma^d (\vx )\leq C
  \quad
  \text{and}
  \quad
   \Big|\frac{1}{d}\,\Im \sigma^d(\vx )\Big|\leq C \qquad \forall \vx \in \Sigma^d,
\end{equation}
for some constants  $c$ and $C$.
}

\medskip

Our first result concerns the existence of a solution for the problem~\eqref{eq:weaka}  when  $d>0$:
\begin{theorem}
  \label{thm:existence}
Under the assumption above, system \eqref{eq:weaka} has a unique solution in $X^d$ for all
  $\vJa \in L^2(\Omega;\C^3)$ and for all $d>0$.
\end{theorem}
The proof of Theorem~\ref{thm:existence} is presented in
Section~\ref{sec:existence}. Next, we state our main convergence result for
the homogenization limit that was formally discussed in
Section~\ref{subsec:mainresult}.
\begin{theorem}
  \label{thm:hom}
  Under the  assumptions of Theorem \ref{thm:existence}, we also assume
  that
  \begin{align*}
    \e^d(\vx) = \e(\vx,\vx/d),
    \quad\text{and}\quad
    \sigma^d(\vx) =d\,\sigma(\vx,\vx/d),
  \end{align*}
  where $\e(\vx,\vy)$ and $\sigma(\vx,\vy)$ are complex-valued scalar
  functions, periodic with respect to $\vy\in Y$, and $\sigma$ is Lipschitz
  continuous:
 \begin{equation}
       \label{eq:hyp4}
 \;|\na_x \sigma(\vx,\vy)|,\;|\na_y \sigma(\vx,\vy)|\leq C
    \qquad\forall\,(\vx,\vy) \in \Omega\times Y.
\end{equation}
  Then, for all $\vJ_a\in L^2(\Omega;\C^3)$, the electric field
  $\vE^d(\vx)$ of \eqref{eq:weaka} and the corresponding magnetic field
  $\vH^d (\vx)$ defined by \eqref{eq:mag} converge weakly in
  $L^2(\Omega;\C^3)$ to functions $\vME(\vx)$ and $\vMH(\vx)$
  satisfying
  the following weak form of the homogenized system \eqref{eq:maxwellhomogenized}:
  \begin{equation}
    \label{eq:weaka0}
    \vME \in X^0,
    \qquad
    a^0(\vME,\vPsi) = \int_{\Omega } i\omega \vJa \cdot
    \overline\vPsi\, \dx  \qquad \forall\, \vPsi \in X^0,
  \end{equation}
  and\vspace{-1em}
  \begin{align*}
    \vMH (\vx)= \frac{1}{i\omega\mu} \nabla\times \vME(\vx).
  \end{align*}
  Here, the sesquilinear form $a^0:X^0\times X^0\to \C$ is given by
  \begin{align*}
    a^0(\vu,\vv) &=
    \int_{\Omega  } \frac 1 {\mu} (\nabla\times \vu)\cdot (\nabla\times
    \overline \vv) \, \dx - \int_{\Omega } \omega^2 \eff  \vu \cdot
    \overline \vv\, \dx  -\int_{\partial \Omega} i\omega\lambda
    \vu_T \cdot  \overline\vv_T\, \dox,
  \end{align*}
  and the space $X^0$ is defined by
  \begin{align*}
    X^0 =\big\{ \vu \in H(\curl;\Omega) \,;\, \vu_T\in
    L^2(\partial\Omega;\C^3) \big\}.
  \end{align*}
  The effective permittivity $ \vec \eff$ is given by
  \eqref{eq:effectiveperm} (with $P_T=
  \text{diag}\,\big(1,1,0\big)\in\R^{3\times3}$).
\end{theorem}

Note that $\sigma(\vx,\vy)$ is only really defined for $\vy\in \Sigma_0$,
but in our simple geometry we can treat it as a function defined for all
$\vy\in Y$ which is independent of $\vy_3$. In particular the gradient
$\na_y\sigma$  in \eqref{eq:hyp4} is a tangential gradient. In the
framework of this theorem, conditions \eqref{eq:hyp1}-\eqref{eq:hyp1'} are
equivalent to
\begin{equation}
    \label{eq:hyp3}
    c \leq \Im \e(\vx,\vy),\; \Re \sigma (\vx,\vy)\leq C,\quad    |\e(\vx,\vy)|,\;|\sigma(\vx,\vy)| \leq C
    \quad\forall\,(\vx,\vy) \in \Omega\times Y .
\end{equation}
The proof of Theorem~\ref{thm:hom} is deferred to
Section~\ref{sec:convergence}.
\begin{remark}\label{rem:matrix}
  Theorems~\ref{thm:existence} and \ref{thm:hom} also hold when $\mu$,
  $\e^d$ and $\sigma^d$ are matrix valued, provided all the respective
  assumptions are replaced with appropriate matrix inequalities. For
  example, the first conditions in \eqref{eq:hyp1}-\eqref{eq:hyp1'} should be replaced by
  \begin{align*}
     c |\xi|^2\leq  \Im ( \e^d \xi \cdot \overline \xi)  \leq C |\xi|^2 ,
     \qquad
     c |\xi|^2\leq \frac 1 d\Re ( \sigma^d \xi \cdot \overline \xi) \leq C |\xi|^2,
     \qquad \forall \xi\in\R^3.
  \end{align*}
\end{remark}

As a last step in this section, we establish the well-posedness of
\eqref{eq:weaka0}.

\begin{theorem}
  \label{thm:existence0}
  Under the assumptions of Theorem \ref{thm:hom}, the homogenized problem
  \eqref{eq:weaka0} has a unique solution  $\vME(\vx)$ in $X^0$ for all
  $\vJa \in L^2(\Omega;\C^3)$.
\end{theorem}
The proof of this theorem is deferred to the end of
Section~\ref{sec:existence}.


\section{Proofs of Theorem \ref{thm:existence} and Theorem \ref{thm:existence0}}
\label{sec:existence}

In this section, we prove the existence of a solution for both the weak
formulation of the problem for finite microstructure scale, $d>0$, as well
as for homogenized problem \eqref{eq:weaka0}. The proof is based on the
Lax-Milgram theorem, which we recall here for the convenience of the
reader:
\begin{theorem}[Lax-Milgram Theorem]\label{thm:LM}
  Let $X$ be a Hilbert space over $\C$ and let $b:X\times X\to \C$ be a
  sesquilinear form such that
  \begin{align*}
    \begin{aligned}
      &b \text{ is bounded:}
      &\quad
      |b(u,v)| &\leq C_1 \| u\|_X\|v\|_X
      \qquad
      &\forall u,\, v\in X,
      \\
      &b \text{ is coercive:}
      &\quad
      \Re b(u,u) &\geq c_2 \|u\|_X^2
      \qquad
      &\forall u\in X.
    \end{aligned}
  \end{align*}
  Then, for any $L:X\to \C$ linear and bounded, there exists a unique $u\in
  X$ such that
  \begin{align*}
    b(u,\Psi) = L(\Psi) \qquad\forall \Psi \in X.
  \end{align*}
\end{theorem}

\begin{proof}[Proof of Theorem~\ref{thm:existence}]
  In order to apply Theorem~\ref{thm:LM}, we first replace the permeability
  $\mu\in \R$ by $\mu + i\eta$ for some small $\eta>0$ and consider the
  sesquilinear form $b^d_\eta(u,v)=ia^d(u,v)$. Indeed, we have
  \begin{align*}
    \Re b^d_\eta(\vu,\vu) & =
    \frac {\eta} {\mu^2+\eta^2} \int_{\Omega  } |\nabla\times \vu|^2 \, \dx
    + \int_{\Omega}  \omega^2 \Im(\e^d) |\vu|^2  \dx\\
    &\qquad\qquad + d \int_{\Sigma^d} \omega\,\frac 1d \Re( \sigma^d)
    |\vu_T|^2 \, \dox
    +\int_{\partial \Omega}  \omega\lambda  |\vu_T|^2 \, \dox,
  \end{align*}
  and so assumptions~\eqref{eq:hyp1}-\eqref{eq:hyp1'} implies
  \begin{align*}
    \Re b^d_\eta(\vu,\vu)  \geq c_\eta \| \vu\|_{X^d}^2 \qquad \forall \vu\in
    X^d,
  \end{align*}
  for some small constant $c_\eta>0$ depending on $\eta$ (and $d$).
  Furthermore, it is readily seen that assumptions \eqref{eq:hyp1}-\eqref{eq:hyp1'} also gives
  \begin{align*}
    | b^d_\eta(\vu,\vv)| \leq C  \| \vu\|_{X^d} \| \vv\|_{X^d}.
  \end{align*}
  Hence, the Lax-Milgram theorem (Theorem \ref{thm:LM}) implies the existence
  of a unique $\vu^\eta$ solution of
  \begin{equation}
    \label{eq:weakb}
      \vu^\eta \in X^d,
      \qquad
      b^d_\eta(\vu^\eta,\vv) = - \int_{\Omega } \omega \vJa \cdot
      \overline\vv\, \dx \qquad \forall \vv \in X^d.
  \end{equation}
  Next, we write
  \begin{align*}
    b^d_\eta(\vu^\eta,\vu^\eta) = - \int_{\Omega } \omega \vJa \cdot
    \overline\vu^\eta \, \dx,
  \end{align*}
  and by taking the real and imaginary parts of this relation, we obtain
  \begin{multline}
    \label{eq:bdimag}
    \int_{\Omega} \frac \eta {\mu^2+\eta^2} |\nabla\times \vu^\eta|^2\, \dx
    +\int_{\Omega } \omega^2 (\Im \e^d)  |\vu^\eta|^2\, \dx
    \\
    + d \int_{\Sigma^d} \omega \big(\frac 1 d\,\Re \sigma^d\big)
    |\vu^\eta_T|^2 \, \doy
    +\int_{\partial \Omega} \omega\lambda |\vu^\eta_T|^2\,\dox
    = - \int_{\Omega } \omega\, \Re (\vJa\cdot{ \overline \vu^\eta}) \, \dx
  \end{multline}
  and
  \begin{multline}
    \label{eq:bdreal}
    \int_{\Omega} \frac {\mu} {\mu^2+\eta^2} |\nabla\times \vu^\eta|^2\, \dx -
    \int_{\Omega} \omega^2 (\Re \e^d) |\vu^\eta |^2 \, \dx
    \\
    + d \int_{\Sigma^d} \omega \big(\frac 1 d\,\Im \sigma^d\big)
    |\vu^\eta_T|^2\, \dox
    = \int_\Omega \omega\, \Im (\vJa\cdot\overline \vu^\eta) \, \dx.
  \end{multline}
  Using \eqref{eq:hyp1}-\eqref{eq:hyp1'} and Young's inequality, we notice that
  \eqref{eq:bdimag} gives
  \begin{align}
    \label{eq:bdres1}
    \| \vu^\eta\|^2_{L^2(\Omega)} +  \| \vu^\eta_T\|^2_{L^2(\Sigma^d)} + \|
    \vu^\eta_T\|^2_{L^2(\partial \Omega)}  \leq C \|
    \vJa\|_{L^2(\Omega)}^{2},
  \end{align}
  and the use of~\eqref{eq:hyp1}-\eqref{eq:hyp1'} and~\eqref{eq:bdreal} combined
  with~\eqref{eq:bdres1} in turn implies
  \begin{align*}
    \|\nabla\times \vu^\eta\|^2_{L^2(\Omega)}  \leq C \|
    \vJa\|_{L^2(\Omega)}^{2}
  \end{align*}
  for some constant $C$ independent of $\eta$. Thus, we can pass to the
  limit $\eta\to 0$ in \eqref{eq:weakb}. This result completes the proof of
  Theorem \ref{thm:existence}.
\end{proof}

The corresponding result for the homogenized system can be established in a
similar vein.

\begin{proof}[Proof of Theorem~\ref{thm:existence0}]
  The well-posedness of \eqref{eq:weaka0} depends  on the properties of the
  effective permittivity, $\eff$. We recall that formula
  \eqref{eq:effectiveperm} gives
  \begin{multline*}
    \eff_{ij} = \int_Y\e(\vx,\vy)(\ve_ j +\nabla_y\vec\chi_j(\vx,\vy))\cdot
    \ve_i\dy
    \\
    -\frac1{i\omega} \int_{\Sigma_0}\big(\sigma(\vx,\vy)
    (\ve_j+\nabla_y\vec\chi_j(\vx,\vy))_T \cdot {\ve_i}_T \doy,
  \end{multline*}
  where $\vec\chi(\vx,\vy)$ solves the cell problem \eqref{eq:cell-intro}.
  Furthermore, by multiplying the cell problem by $\overline {\vec \chi_i}$ and
  integrating over $Y$ (see also the weak formulation
  \eqref{eq:cellweak}), we obtain
  \begin{multline*}
    \int_Y \e(\vx,\vy) (\ve_j +\nabla_y \chi_j(\vx,\vy)) \cdot \nabla_y
    \overline {\vec \chi_i}(\vy)\,\dy\,
    \\
    - \frac{ 1} { i \omega} \int_{\Sigma_0}\sigma(\vx,\vy)  \big(\ve_j +\nabla_y
    \vec\chi_j(\vx,\vy)\big)_T\cdot \big(\nabla_y \overline {\vec\chi_i}
    (\vy)\big)_T \, \doy =0.
  \end{multline*}
  Combining the last two formulas, we arrive at the following alternative
  formula for $\eff$:
  \begin{multline}
    \label{eq:effectiveperm1}
    \eff_{ij} = \int_Y\e(\vx,\vy)(\ve_ j +\nabla_y\vec\chi_j(\vx,\vy))\cdot (
    \ve_i +\nabla_y\overline{\vec\chi_i}(\vx,\vy))\dy
    \\
    -\frac1{i\omega} \int_{\Sigma_0} \big( \sigma(\vx,\vy)
    (\ve_j+\nabla_y\vec\chi_j(\vx,\vy))\big)_T \cdot \big(\ve_i + \nabla_y
    \overline {\vec\chi_i}(\vx,\vy)\big)_T \doy.
  \end{multline}
  In this form, we see that when  $\e$ and $\sigma$ are scalar (as in Theorem
  \ref{thm:hom}), we have
  \begin{align*}
    \eff \vec \xi\cdot \overline {\vec \xi} & = \int_Y\e(\vx,\vy)
    \biggl|\vec \xi+\sum_{j=1}^3 \na_{y} \chi_j(\vx,\vy)
    \xi_j\biggr|^2\dy
    \\
    &\qquad - \frac{ 1} { i \omega} \int_{\Sigma_0}\sigma(\vx,\vy)
    \biggl|\Big(\vec \xi+\sum_{j=1}^3 \na_{y} \chi_j(\vx,\vy) \xi_j
    \Big)_T\biggr|^2\dy,  \qquad \forall \vec \xi\in \R^3,
  \end{align*}
  and thus assumption \eqref{eq:hyp3} implies the inequality
  \begin{align}\label{eq:hyp5}
    \Im \big( \eff(\vx) \vec \xi\cdot \overline {\vec \xi} \big) \geq
    c |\vec\xi|^2.
  \end{align}
  We proceed exactly as in the proof of Theorem \ref{thm:existence} (with
  inequality \eqref{eq:hyp5} playing the role of assumption
  \eqref{eq:hyp1}-\eqref{eq:hyp1'}). This concludes the proof of
  Theorem~\ref{thm:existence0}.
\end{proof}


\section{Prerequisites and proof of Theorem \ref{thm:hom}}
\label{sec:convergence}

The core of this section is devoted to the two-scale convergence needed to
establish our main homogenization result. In order to pass to the limit
in~\eqref{eq:weaka} and prove Theorem~\ref{thm:hom}, we must first
establish an important a priori estimate as explained below
(Section~\ref{subsec:apriori-est}). The related machinery of two-scale
convergence is utilized in Section~\ref{subsec:2-sc}. This leads to the
proof of a main proposition for the homogenized system
(Proposition~\ref{prop:hom} in Section~\ref{subsec:homog-sys}); and the
extraction of the requisite cell problem (Section~\ref{subsec:cell-hom}).
We conclude this section by proving Theorem~\ref{thm:hom}; see
Section~\ref{subsec:cell-hom}.

\subsection{An a priori estimate}
\label{subsec:apriori-est}

We start with the following proposition, which establishes a crucial
\emph{uniform} a priori estimate that is used throughout the proof of this
section. We recall that the norm $\|\cdot\|_{L^2(\Sigma^d)}$ is defined by
\eqref{eq:Sigmanorm} and includes a factor equal to~$d$.
\begin{proposition}
  \label{prop:estimates}
  Under the assumptions of Theorem \eqref{thm:hom}, the solution $\vE^d$ of
  \eqref{eq:weaka} is bounded in $X^d$ \emph{uniformly with respect to
  $d$}. More precisely, there exists a constant $C$ independent of $d$ such
  that
  \begin{equation}\label{eq:estimate}
    \frac{1}{\omega}  \|\nabla\times
    \vE^d\|^2_{L^2(\Omega)} +\omega \| \vE^d\|^2_{L^2(\Omega)}+ \|
    \vE^d_T\|^2_{L^2(\Sigma^d)} +  \| \vE^d_T\|^2_{L^2(\partial \Omega)}
    \leq C \| \vJ_a\|_{L^2(\Omega)}.
  \end{equation}
\end{proposition}
\begin{proof}
  We take $\vPsi= {\vE^d}$ in \eqref{eq:weaka}, to obtain
  \begin{multline}
    \label{eq:energy0}
    \int_{\Omega  } \frac 1 {\mu} |\nabla\times \vE^d|^2\, \dx
    =
    \omega^2 \int_{\Omega } \e(\vx,\vx/d)  |\vE^d |^2 \, \dx
    + i\omega d \int_{\Sigma^d}  \sigma(\vx,\vx/d) |\vE^d_T|^2\, \dox
    \\
    + i\omega\lambda \int_{\partial \Omega} |\vE^d_T|^2\,\dox
    + i\omega\int_{\Omega } \vJa \cdot  \overline{\vE^d}\, \dx.
  \end{multline}
  Taking the real and imaginary part of \eqref{eq:energy0} gives,
  respectively,
  \begin{multline}
    \label{eq:energyreal}
    \int_{\Omega} \frac 1 {\mu} |\nabla\times \vE^d|^2\, \dx
    =
    \omega^2  \int_{\Omega}  \Re \e(\vx,\vx/d) |\vE^d |^2 \, \dx
    \\
    \quad - \omega d \int_{\Sigma^d}  \Im \sigma(\vx,\vx/d) |\vE^d_T|^2\, \dox
    - \omega \int_\Omega \Im (\vJa\cdot\overline \vE^d) \, \dx,
  \end{multline}
  and
  \begin{multline}
    \label{eq:energyimag}
    \omega^2 \int_{\Omega } \Im \e(\vx,\vx/d)  |\vE^d|^2\, \dx
    +\omega   d \int_{\Sigma^d}  \Re \sigma (\vx,\vx/d) |\vE^d_T|^2 \, \doy
    \\
    +\omega\lambda \int_{\partial \Omega}  |\vE^d_T|^2\,\dox
    = - \omega  \int_{\Omega}
    \Re (\vJa\cdot\overline \vE^d) \, \dx.
  \end{multline}
  By invoking~\eqref{eq:hyp3} and Young's inequality, we derive from
  \eqref{eq:energyimag} that
  \begin{align*}
   \omega \int_{\Omega } |\vE^d|^2\, \dx
    +  d \int_{\Sigma^d}  |\vE^d_T|^2 \, \doy
    +\lambda \int_{\partial \Omega}  |\vE^d_T|^2\,\dox
    \leq C \int_{\Omega} |\vJa|^2 \, \dx.
  \end{align*}
  Furthermore, by virtue of~\eqref{eq:energyreal} and~\eqref{eq:hyp4} we infer that
  \begin{align*}
    \int_{\Omega} \frac 1 {\mu} |\nabla\times \vE^d|^2\, \dx
    \leq C\omega \int_{\Omega} |\vJa|^2 \, \dx.
  \end{align*}
  The last two inequalities assert the desired result.
\end{proof}

\subsection{Two-scale convergence}
\label{subsec:2-sc}

The proof of our homogenization result~\eqref{eq:maxwellhomogenized} relies
on the well-known notion of two-scale convergence (see
\cite{allaire,nguetseng}). Several important results related to two-scale
convergence are summarized in \ref{app:two-scale} for the convenience of
the reader. Recall that the space $L^2_\#(Y;\C^k)$ (respectively
$H^1_\#(Y;\C^k)$) denotes the closure under the $L^2$ norm (respectively
$H^1$ norm) of the set $\C^\infty_\#(Y;\C^k)$ of smooth $Y$-periodic
functions defined on $\R^3$ with values in $\C^k$ (with $k=1$ or $3$).

Let $\vE^d(\vx)$ be the unique solution of \eqref{eq:weaka} as established
by Theorem \ref{thm:existence}. We define $\vH^d(\vx)$ by~\eqref{eq:mag}
for $\vx\in\Omega$. Proposition~\ref{prop:estimates} implies in particular
that $\vH^d\in L^2(\Omega;\C^3)$. The weak formulation of \eqref{eq:mag}
thus reads
\begin{align}
  \label{eq:weak1}
  \int_{\Omega} \vE^d\cdot (\nabla\times \overline \vPsi) \, \dx
  =
  i\omega\mu \int_{\Omega} \vH^d \cdot  \overline \vPsi\, \dx \qquad
  \forall \vPsi\in \HT,
\end{align}
where $\HT = \{\vu\in H(\curl;\Omega)\,;\, \vu_T = 0 \mbox{ on } \pa\Omega
\}.$ Furthermore, in view of~\eqref{eq:mag}, we can now  write \eqref{eq:weak} as
\begin{multline}
  \label{eq:weak2}
  \int_{\Omega}  \vH^d \cdot (\nabla\times \overline \vPsi) \, \dx
  - \int_{\partial \Omega} \lambda \vE^d \cdot  \overline\vPsi_T\, \dox
  +\int_{\Omega}i\omega\e(\vx,\vx/d) \vE^d\cdot\overline \vPsi\, \dx
  \\
  = d\int_{\Sigma^d}\sigma(\vx,\vx/d)\vE^d_T\cdot \overline \vPsi_T\, \dox
  + \int_{\Omega} \vJa\cdot\overline\vPsi\,\dx
  \qquad\forall \vPhi \in X^d.
\end{multline}
Our goal is to pass to the limit $d\to 0$ in~\eqref{eq:weak1} and
\eqref{eq:weak2}. We start with the following lemma.
\begin{lemma}
  \label{lem:ts}
  Up to a subsequence, the functions $\vE^d(\vx)$ and $\vH^d(\vx)$
  two-scale converge to functions $\vE^{(0)}(\vx,\vy)$ and
  $\vH^{(0)}(\vx,\vy)$ in $L^2(\Omega;L^2_\#(Y;\C^3))$ which satisfy
  \begin{align*}
    \vE^{(0)}(\vx,\vy) &= \vME(\vx) + \nabla_y\varphi(\vx,\vy),
    \\
    \vH^{(0)}(\vx,\vy) &= \vMH(\vx),
  \end{align*}
  for some functions $\vME$, $\vMH\in  L^2(\Omega;\C^3)$ and
  $\varphi(\vx,\vy)\in L^2(\Omega;H^1_\#(Y,\C))$.
\end{lemma}
\begin{proof}
  A proof of the Lemma can also be found in
  \cite{wellander2001,wellander2003}. Proposition \ref{prop:estimates}
  implies that the sequences $\vE^d$ and $\vH^d$ are both bounded in
  $L^2(\Omega;\C^3)$. The classical two-scale convergence result
  (Theorem~\ref{thm:ts}) thus implies that there exist some subsequences,
  still denoted $\vE^d$ and $\vH^d$, which two-scale converge to
  $\vE^{(0)}(\vx,\vy)$ and $\vH^{(0)}(\vx,\vy)$, respectively.

  We now consider a test function $\vPsi^d(\vx) =d w(\vx)\vPhi(\vx/d)$
  with $\vPhi\in C^\infty_\#(Y;\C^3)$ and $w\in \mathcal D(\Omega)$. Since
  $\vPsi^d=0$ on $\pa\Omega$, integration by parts entails
  \begin{align*}
    \int_\Omega \nabla \times \vE^d \cdot \vPsi^d\dx
    &=
    \int_\Omega \vE^d \cdot \nabla \times \vPsi ^d  \dx
    \\
    &=
    \int_\Omega \vE^d(\vx) \cdot \left[  d \nabla w (\vx) \times
    \vPhi(\vx/d) + w(\vx) \nabla \times \vPhi(\vx/d) \right]\dx.
  \end{align*}
  This equation can be re-arranged as
  \begin{multline*}
    \int_\Omega  w(\vx) \vE^d(\vx) \cdot  \nabla \times \vPhi(\vx/d)\,\dx
    \\
    =
    d \int_\Omega \nabla \times \vE^d(\vx) \cdot \vPhi(\vx/d)w(\vx) \dx
    - d \int_\Omega \vE^d(\vx) \cdot \nabla w (\vx)\times\vPhi(\vx/d)\dx.
  \end{multline*}
  Passing to the limit ($d\to 0$) via stability
  estimate~\eqref{eq:estimate}, we deduce that
  \begin{equation}
    \label{eq:wep}
    \int_\Omega w(\vx) \int _Y \vE^{(0)}(\vx,\vy) \cdot \nabla_y \times
    \vPhi(\vy) \,\dy \, \dx = 0.
  \end{equation}
  At this stage, introduce the homogenized electric field
  \begin{align*}
    \vME(\vx) = \int_Y \vE^{(0)}(\vx,\vy)\, \dy.
  \end{align*}
  Equation \eqref{eq:wep} implies that the vector-valued function
  $\vf(\vx,\vy) =  \vE^{(0)}(\vx,\vy) - \vME(\vx) \in
  L^2(\Omega;L^2_\#(Y;\C^3)))$ satisfies
  \begin{align*}
    \nabla_y \times \vf (\vx,\vy) = 0\quad\text{in } \mathcal
    D'(\Omega\times Y),\qquad \int_Y \vf(\vx,\vy)\, \dy =0
    \quad\text{for a.\,e. } \vx\in \Omega.
  \end{align*}
  Utilizing a result from Fourier analysis (Lemma~\ref{lem:homf}), we
  conclude that there exists a scalar $\varphi(\vx,\vy)$ in
  $L^2(\Omega;H^1_\#(Y))$ such that $f(\vx,\vy) =
  \nabla_y\varphi(\vx,\vy)$. Thus, we proved the first statement of
  Lemma~\ref{lem:ts}, namely, that
  \begin{align*}
    \vE^{(0)}(\vx,\vy)= \vME(\vx) + \nabla_y\varphi(\vx,\vy).
  \end{align*}
  In order to derive the corresponding result for $\vH^{(0)}$, we first
  note that~\eqref{eq:weak2} and the bounds of Proposition
  \ref{prop:estimates} entail the estimate
  \begin{equation}
    \label{eq:Hbd}
    \left| \int_\Omega \vH^d \cdot \nabla\times \vPsi\, \dx \right|
    \leq
    C\,\|\vJ_a\|_{L^2(\Omega)} \left\{\| \vPsi\|_{L^2(\Omega)} + \| \vPsi_T
    \|_{L^2( \partial\Omega)} + \| \vPsi_T \|_{L^2(\Sigma^d)} \right\}.
  \end{equation}
  By using a test function of the form  $\vPsi^d(\vx) =d w(\vx)\vPhi(\vx/d)
  $ with $\vPhi\in C^\infty_\#(Y;\C^3)$ and $w\in \mathcal D(\Omega)$, by
  analogy to~\eqref{eq:wep} we obtain
  \begin{align*}
    \int_\Omega w(\vx) \int_Y \vH^{(0)}(\vx,\vy) \cdot \nabla_y \times
    \vPhi(\vy)\, \dx\, \dy =0.
  \end{align*}
  By proceeding as above, we now show that there exists a scalar function
  $\varphi_1(\vx,\vy)\in L^2(\Omega;H^1_\#(Y;\C))$ such that
  \begin{equation}
    \label{eq:H0}
    \vH^{(0)}(\vx,\vy)= \vMH(\vx) + \nabla_y\varphi_1(\vx,\vy).
  \end{equation}
  The substitution of test function $\vPsi^d (\vx) = d \nabla (u(\vx)
  v(\vx/d))$ into~\eqref{eq:weak1} implies
  \begin{align*}
    0
    &=
    d \int_{\Omega } \vH^d(\vx) \cdot \nabla (u(\vx) v(\vx/d) ) \, \dx
    \\
    &=
    d \int_{\Omega } v(\vx/d) \vH^d (\vx)\cdot \nabla_x u(\vx) \, \dx +
    \int_{\Omega } u(\vx) \vH^d(\vx) \cdot  (\nabla_y v) (\vx/d) ) \, \dx.
  \end{align*}
  Passing to the limit ($d\to 0$), we observe that
  \begin{align*}
    \int_\Omega u(\vx) \int_Y \vH^{(0)}(\vx,\vy) \cdot \nabla_y v(\vy)\,\
    \dy\,\dx =0.
  \end{align*}
  This equality implies that $\nabla_y\cdot \vH^{(0)}(\vx,\vy) =0$ in the
  sense of distribution. Thus, by combining this result with~\eqref{eq:H0}
  we conclude that
  \begin{align*}
    \Delta_y \varphi_1(\vx,\vy)=0 \mbox{ in } \mathcal D'(Y)
    \quad\text{for a.\,e. } x\in\Omega.
  \end{align*}
  Together with the periodic boundary conditions ($\vy\mapsto
  \varphi_1(\vx,\vy) $ is in $H^1_\#(Y;\C)$) this statement implies that
  $\varphi_1(\vx,\vy)$ is constant and, thus, $\nabla_y \varphi_1 = 0$.
  Equality \eqref{eq:H0} now reduces to the equation $\vH^{(0)}(\vx,\vy)=
  \vMH(\vx)$. This result concludes the proof of Lemma~\ref{lem:ts}.
 \end{proof}

\subsection{Derivation of the homogenized system}
\label{subsec:homog-sys}
We now fix a subsequence $\{d_l\}_{l\in \N}$ such that $\lim_{l\to\infty}
d_l = 0$ and the functions $\vE^{d_l}(\vx)$ and $\vH^{d_l}(\vx)$  two-scale
converge  to $\vE^{(0)}(\vx,\vy) = \vME(\vx) + \nabla_y\varphi(\vx,\vy)$
and $\vH^{(0)}(\vx,\vy) = \vMH(\vx)$ as in Lemma \ref{lem:ts}. We denote $$
\vE^{l}(\vx):=\vE^{d_l}(\vx), \quad \vH^{l}(\vx):=\vH^{d_l}(\vx).$$ We will
then prove that the limits $\vME(x)$ and $\vMH(x)$ solve the homogenized
problem \eqref{eq:maxwellhomogenized}. More precisely, the main result that
we establish in this section is the following:
\begin{proposition}\label{prop:hom}
 The functions
 $\vE^l(\vx)$ and
  $\vH^l(\vx)$ converge weakly in $L^2(\Omega;\C^3)$ to $\vME(x)$ and $\vMH(x)$,
  where $\vME \in X^0$ and $\vMH\in L^2(\Omega;\C^3)$ satisfy
  \begin{multline}
    \label{eq:h1}
    \int_{\Omega} \vMH(\vx) \cdot (\nabla\times \overline \vPsi(\vx)) \, \dx
    -\int_{\partial \Omega} \lambda \vME _T (\vx) \cdot \overline\vPsi_T(\vx)
    \,\dox
    \\
    + \int_{\Omega}\int_Y i \omega \e(\vx,\vy) (\vME(\vx) +\nabla_y
    \varphi(\vx,\vy)) \cdot \overline \vPsi(\vx)\, \dy\, \dx
    \\
    - \int_\Omega\int_{\Sigma_0} \sigma(\vx,\vy) (\vME(\vx) +\nabla_y
    \varphi(\vx,\vy))_T \cdot \overline \vPsi_T(\vx)\, \doy \dx
    = \int_{\Omega} \vJa(\vx) \cdot \overline\vPsi(\vx)\,\dx,
  \end{multline}
  for all $\vPsi\in X^0$, and
  \begin{align}
    \label{eq:h2}
    \int_{\Omega} \vME \cdot (\nabla\times \overline \vPsi) \, \dx
    &=
    i\omega\mu \int_{\Omega} \vMH \cdot \overline \vPsi\, \dx
    &&\forall\,\vPsi\in\HT.
  \end{align}
  Furthermore, the corrector $\varphi(\vx,\vy)\in L^2(\Omega;H^1_\#(Y,\C))$
  satisfies $\nabla_y \varphi(\vx,\vy)_T \in L^2(\Omega\times\Sigma_0)$ and
  is determined by the following problem:
  \begin{multline}\label{eq:h3}
    i\omega\int_Y \e(\vx,\vy)\big(\vME(\vx)
    +\nabla_y \varphi(\vx,\vy)\big) \cdot \nabla_y \overline v(\vy)\,\dy
    \\
    = \int_{\Sigma_0} \sigma(\vx,\vy) \big(\vME(\vx) + \nabla_y
    \varphi(\vx,\vy)\big)_T\cdot \nabla_T \overline v (\vy) \, \doy
    \quad\forall v\in H^1_\#(Y;\C)
  \end{multline}
  and for a.\,e. $x\in \Omega$. In the above, $\Sigma_0$ denotes the
  hypersurface $\{y_3=0\}$ in $Y$.
\end{proposition}
In order to prove Proposition~\ref{prop:hom} we have to pass to the limit
in \eqref{eq:weak1} and \eqref{eq:weak2}, as $d_l\to 0$. While this procedure
is relatively straightforward for \eqref{eq:weak1}, the passage to the
limit in \eqref{eq:weak2} is more delicate because of the presence of
surface integrals on $\Sigma^d$. We will prove the following key
proposition that establishes the limit of the requisite surface integrals:
\begin{proposition}\label{prop:limsum}
  The corrector
  $\varphi$  satisfies $\nabla_y \varphi(\vx,\vy)_T\in
  L^2(\Omega\times\Sigma_0)$, and  for all functions $\vF(\vx,\vy)$ defined
  in $\Omega\times Y$ that are periodic with respect to $y$ and admit
  $\vF,\; \na_x \vF,\; \na_y \vF \in L^\infty(\Omega\times Y)$ there holds:
  \begin{multline}\label{eq:limsum1}
    \lim_{k\to \infty} d_l \int_{\Sigma^{d_l}}  \vE^l_T(\vx) \cdot
    \vF_T(\vx,\vx/{d_l})\, \dox
    =
    \\
    \int_\Omega\int_{\Sigma_0} (\vME(\vx) +\nabla_y
    \varphi(\vx,\vy))_T\cdot F_T(\vx,\vy)\, \doy \dx.
  \end{multline}
\end{proposition}
The proof of this proposition is deferred to Section~\ref{subsec:limsum}.
\begin{proof}[Proof of Proposition~\ref{prop:hom}]
  The choice of the subsequence $\vE^l(\vx)$ and $\vH^l(\vx)$ together with
  classical two-scale convergence results (Theorem~\ref{thm:ts}) imply that
  $\vE^l(\vx)$ and $\vH^l(\vx)$ converge weakly (in $L^2(\Omega;\C^3)$) to
  the functions $\vME(\vx)$ and $\vMH(\vx)$. Furthermore, Proposition
  \ref{prop:estimates} implies that $\vME(\vx)$ is in $X^0$.

  Passing to the limit in \eqref{eq:weak1}, as $d_l\to 0$, we obtain
  \begin{equation}
    \label{eq:weakh1}
    \int_{\Omega} \vME \cdot (\nabla\times \overline \vPsi) \, \dx
    =
    i\omega\mu \int_{\Omega } \vMH \cdot \overline \vPsi\, \dx
    \qquad\forall\,\vPsi\in \HT,
  \end{equation}
  which implies \eqref{eq:h2}.

  Next, the estimate of Proposition \ref{prop:estimates} implies that
  $\nabla\times \vE^l$ is bounded in $L^2(\Omega)$ and thus converges
  weakly in $L^2(\Omega)$ to $\nabla\times \vME$. Furthermore, Proposition
  \ref{prop:estimates} also implies that $\vE^l_T|_{\partial \Omega}$ is
  bounded in $L^2(\partial\Omega)$. In particular, there is a subsequence
  $\vE^{l'}_T|_{\partial \Omega}$ which converges weakly in
  $L^2(\partial\Omega)$ and for any smooth test function $\vPsi$ we have:
  \begin{align*}
    \int_{\pa\Omega} \vE^{l'}_T \cdot (\overline \vPsi \times \nu) \,\dox
    & = \int_\Omega (\nabla\times \vE^{l'}) \cdot \overline \vPsi - \vE^{l'} \cdot (\nabla\times\overline \vPsi )\, dx
    \\
    & \to  \int_\Omega (\nabla\times \vME) \cdot\overline  \vPsi - \vME \cdot (\nabla\times \overline \vPsi) \, dx
    \\
    & =\int_{\pa\Omega} \vME_T \cdot(  \overline\vPsi\times \nu) \,\dox.
  \end{align*}
  We deduce that this weak limit is $ \vME_T|_{\partial\Omega}$. Since the
  limit is independent of the subsequence $l'$, the whole sequence
  $\vE^{l}_T|_{\partial \Omega}$ converges weakly, viz.,
  \begin{align*}
    \vE^l_T|_{\partial\Omega} \rightharpoonup \vME_T|_{\partial\Omega}
    \quad\text{weakly in }  L^2(\partial\Omega).
  \end{align*}

  We can now pass to the limit in \eqref{eq:weak2}. By invoking the usual
  properties of two-scale convergence in combination with the results of
  Proposition \ref{prop:limsum} and the choice $\vF(\vx,\vy) =
  \sigma(\vx,\vy) \overline \vPsi (\vx,\vy)$, where $\vPsi(\vx,\vy)$ is a
  smooth vector-valued test function, we deduce that
  \begin{multline*}
    \int_{\Omega} \vMH(\vx) \cdot (\nabla\times \overline \vPsi(\vx)) \, \dx
    -\int_{\partial \Omega} \lambda \vME _T (\vx) \cdot \overline\vPsi_T(\vx)
    \,\dox
    \\
    + \int_{\Omega}\int_Y i \omega \e(\vx,\vy) (\vME(\vx) +\nabla_y
    \varphi(\vx,\vy)) \cdot \overline \vPsi(\vx)\, \dy\, \dx
    \\
    = \int_\Omega\int_{\Sigma_0} \sigma(\vx,\vy) (\vME(\vx) +\nabla_y
    \varphi(\vx,\vy))_T \cdot \overline \vPsi_T(\vx)\, \doy \dx
    + \int_{\Omega} \vJa(\vx) \cdot \overline\vPsi(\vx)\,\dx,
  \end{multline*}
  which in turn gives \eqref{eq:h1}.

  Finally, by taking $\vPsi (\vx) = d_l \nabla (u(\vx) v(\vx/d_l) )$ as a
  smooth test function in \eqref{eq:weak2}, as well as by using
  Proposition~\ref{prop:limsum} with $\vF(\vx,\vy) = u(\vx)\na_y v(\vy)$ we
  pass to the limit (as $d_l\to 0$):
  \begin{multline}\label{eq:cellweak0}
    \int_{\Omega }  \overline u(x) \int_Y i \omega  \e(\vx,\vy) (\vME(\vx)
    +\nabla_y \varphi(\vx,\vy))   \cdot  \nabla_y  \overline v(\vy)\,\dy\,
    \dx
    \\
    = \int_\Omega  \overline u(x)  \int_{\Sigma_0}  \sigma(\vx,\vy) (\vME(\vx)
    +\nabla_y \varphi(\vx,\vy))_T\cdot \nabla_T  \overline v (\vy) \, \doy\,
    \dx,
  \end{multline}
  which is the weak formulation of cell problem \eqref{eq:h3}. This
  concludes the proof of Proposition~\ref{prop:hom}.
\end{proof}

\subsection{Proof of Proposition~\ref{prop:limsum}}
\label{subsec:limsum}
To simplify the notations in this proof, we drop the $l$ dependence of the
subsequence, although it is understood that all limits are taken along the
subsequence $d_l$.

Recalling our notation $\vx=(\vx',x_3)\in\Sigma\times\Gamma$, we note that
\begin{align*}
  d \int_{\Sigma^d} \vE^{d}_T(\vx)\cdot \vF_T(\vx,\vx/{d})\, \dox
  &=
  \sum_{l\in \Gamma^d} d\int_\Sigma \vE^d_T(\vx',kd) \cdot
  \vF_T(\vx',kd;\vx'/d,k) \, \dx'
  \\
  &=
  \sum_{k\in \Gamma^d} d\int_\Sigma \vE^d_T(\vx',kd) \cdot
  \vF_T(\vx',kd;\vx'/d,0) \, \dx'.
\end{align*}
Here, the second equality stems from the fact that $\vec F_T$ is
$y_3$-periodic with period 1. We now introduce the function
\begin{align*}
  \alpha^d(t)
  =
  \sum_{k\in \Gamma^d} d\int_\Sigma \vE^d_T(\vx',(k+t)d) \cdot
  \vF_T(\vx',(k+t)d;\vx'/d,0) \, \dx',
\end{align*}
defined for $t\in [0,1)$. We point out that we only have $\vE^d_T\in
H^{-1/2}(\Sigma)$, so the integral above is not well defined in the
classical sense. However, $F$ is a smooth test function, so we can make
sense of the integral as a duality bracket $\langle \cdot,
\cdot\rangle_{H^{-1/2},H^{1/2}}$. The limit that we need to characterize in
order to prove Proposition \ref{prop:limsum} is $\lim_{d\to0}\alpha^d(0)$.
The desired result will ensue from the following two lemmas.
\begin{lemma}\label{lem:alpha1}
  Under the assumptions of Proposition \ref{prop:limsum}, the function
  $\alpha^d$ satisfies
  \begin{align}
    \|\alpha^d\|_{L^2(0,1)} & \leq C \|\vF\|_{L^\infty (\Omega) } \|
    \vE^d\|_{L^2(\Omega)} \label{eq:ineqalpL2}
    \\
    &  \leq C\,\|\vF\|_{L^\infty (\Omega) } \,
    \|\vJ_a\|_{L^2(\Omega)}\nonumber
  \end{align}
  and is thus bounded in $L^2(0,1)$. Furthermore, when $d\to0$, it converges
  weakly in $L^2(0,1)$
  to the function
  \begin{align*}
    \alpha^0(t):=
    \int_\Omega \int_{Y'} \vE^{(0)}_T(\vx;\vy',t)\cdot
    \vF_T(\vx;\vy',0)\,\dy'\dx,
  \end{align*}
  where $Y'=[0,1]^2$.
\end{lemma}
\begin{lemma}\label{lem:alpha2}
  Under the assumptions of Proposition \ref{prop:limsum}, there exists a
  constant $C$ (depending on $F$, but independent of $d$) such that
  \begin{align}
     \left\| \frac{d \alpha^d}{dt}\right\|^2_{L^2(0,1)}& \leq C (1+d^2) \| \vE^d\|^2_{L^2(\Omega)} \nonumber\\
     &\qquad +
     C d^2( \| \na\times \vE^d\|^2_{L^2(\Omega)}
     +\|\vE^d\|^2_{L^2(\pa\Omega)})\label{eq:alpbdder}\\
     & \leq C (1+d^2) \|\vJ_a\|^2_{L^2(\Omega)}\nonumber
  \end{align}
\end{lemma}
\begin{proof}[Proof of Proposition~\ref{prop:limsum}]
  Deferring the proofs of the last two lemmas to the end of the present
  section, we note that Lemma~\ref{lem:alpha2} implies that $\alpha^d(t)$
  is bounded in $C^{1/2}(0,1)$ by virtue of the Sobolev embedding theorem
  and thus the convergence established in Lemma \ref{lem:alpha1} is
  uniform. In particular, $\alpha^0(t)$ is defined pointwise with $$
  \alpha^0(0)=  \int_\Omega\int_{\Sigma_0} (E^{(0)}(\vx,\vy))_T\cdot
  (F(\vx,\vy))_T\, \doy \dx$$ (recall that $\Sigma_0 = Y'\times\{0\}$)  and
  \begin{align*}
    \lim_{d\to 0}\alpha^d(0) =   \int_\Omega\int_{\Sigma_0}
    (E^{(0)}(\vx,\vy))_T\cdot (F(\vx,\vy))_T\, \doy \dx,
  \end{align*}
  which is \eqref{eq:limsum1}. Note also that
  \begin{align*}
    \alpha^d(0)
    & = d \int_{\Sigma^d}  \vE^d_T(\vx) \cdot \vF_T(\vx,\vx/d)\, \dox \leq
    C \left( d \int_{\Sigma^d} | \vF_T(\vx,\vx/d)|^2 \, \dox \right)^{1/2}.
  \end{align*}
  Passing to the limit of this expression (using Lemma  2.4 in
  \cite{Allaire95} and noting that under the assumptions of
  Proposition~\ref{prop:limsum} we have $\vF_T \in W^{1,\infty} (\Omega
  \times Y)\subset C(\Omega\times Y)$), we deduce that
  \begin{align*}
    \alpha^0(0) \leq C \left(  \int_{\Omega}\int_{\Sigma_0} |
    \vF_T(\vx,\vy)|^2 \, \doy dx \right)^{1/2}.
  \end{align*}
  It follows that $(E^{(0)}(\vx,\vy))_T\in    L^2(\Omega\times \Sigma_0)$
  and thus $\na_y\varphi(\vx,\vy)_T \in L^2(\Omega\times \Sigma_0)$, which
  completes the proof of Proposition \ref{prop:limsum}.
\end{proof}

We return to the task of proving Lemmas \ref{lem:alpha1} and
\ref{lem:alpha2}.
\begin{proof}[Proof of Lemma~\ref{lem:alpha1}]
  Note that to prove \eqref{eq:ineqalpL2}, we can first assume that $\vE^d$
  is smooth enough for the all the integrals below to make sense and
  conclude that \eqref{eq:ineqalpL2} holds by a density argument. We have:
  \begin{align*}
    \int_0^1 |\alpha^d(t)|^2\, dt
    & \leq C\,\| \vF\|^2_{L^\infty }\int_0^1 \left| \sum_{k\in \Gamma^d}
    d\int_\Sigma |\vE^d_T(\vx',(k+t)d)| \, \dx'\right|^2\dt
    \\
    & \leq C\,\| \vF\|^2_{L^\infty } |\tilde \Gamma^d| \int_0^1 \sum_{k\in \Gamma^d}
    d \left|\int_\Sigma | \vE^d_T(\vx',(k+t)d)|\, \dx'\right|^2\dt
    \\
    & \leq C\,\| \vF\|^2_{L^\infty } |\Gamma||\Sigma|\sum_{k\in \Gamma^d} d
    \int_0^1 \int_\Sigma \big| \vE^d_T(\vx',(k+t)d)\big|^2\, \dx' \dt\\
    & \leq C\,\| \vF\|^2_{L^\infty }|\Gamma||\Sigma|\sum_{k\in \Gamma^d}
    \int_{kd}^{(k+1)d} \int_\Sigma \big| \vE^d_T(\vx',x_3)\big|^2\, \dx' \,
    \dx_3
    \\
    & \leq C\,\| \vF\|^2_{L^\infty }|\Gamma||\Sigma| \int_{\tilde \Gamma^d}
    \int_\Sigma \big| \vE^d_T(\vx',x_3)\big|^2\, \dx' \, \dx_3.
  \end{align*}
  Recall that $\tilde\Gamma^d = \bigcup_{k\in \Gamma^d}\big[kd,(k+1)d\big]$
  satisfies $(-L+d,L-d)\subset \tilde\Gamma^d\subset \Gamma = (-L,L)$. This
  implies \eqref{eq:ineqalpL2} and using Proposition~\ref{prop:estimates} we get
  \begin{align*}
    \|\alpha^d\|_{L^2(0,1)} \leq C \|
    \vF\|_{L^\infty } \| \vE^d\|_{L^2(\Omega)} \leq C\|
    \vF\|_{L^\infty }\,
    \|\vJ_a\|_{L^2(\Omega)}.
  \end{align*}
  The last inequality entails in particular that, up to another
  subsequence, $\alpha^d$ converges weakly in $L^2(0,1)$ to a function
  $\alpha^0$. Furthermore, for any $1$-periodic test function
  $\varphi:\R\to \C$, we assert that
  \begin{align*}
    \int_0^1 \alpha^d(t)\varphi(t)\dt &  = \sum_{k\in \Gamma^d} d
    \int_0^1 \int_\Sigma   \vE^d_T(\vx',(k+t)d) \cdot
    \vF_T(\vx',(k+t)d;\vx'/d,0)\varphi(t) \, \dx' \dt
    \\
    &=
    \sum_{k\in \Gamma^d} \int_{kd}^{(k+1)d} \int_\Sigma \vE^d_T(\vx',x_3)
    \cdot\vF_T(\vx',x_3;\vx'/d,0)\varphi(x_3/d) \, \dx' \, \dx_3
    \\
    &=
    \int_{\tilde \Gamma^d} \int_\Sigma\vE^d_T(\vx',x_3)
    \cdot\vF_T(\vx',x_3;\vx'/d,0)\varphi(x_3/d) \, \dx' \, \dx_3
    \\
    &=
    \int_{\Sigma\times \tilde \Gamma^d}\vE^d_T(\vx) \cdot
    \vF_T(\vx;\vx'/d,0)\varphi(x_3/d)\,\dx.
  \end{align*}
  Using the definition of two-scale convergence and \eqref{eq:Gamma}, we
  see that
  \begin{align*}
    \lim_{d\to 0} \int_0^1 \alpha^d(t)\varphi(t)\dt & = \int_{\Omega}
    \int_Y \vE^{(0)}_T(\vx,\vy) \cdot \vF_T(\vx;\vy',0)\,\varphi(y_3)\,\dy
    \,\dx
    \\
    &=
    \int_0^1\left( \int_{\Omega} \int_{Y'} \vE^{(0)}_T(\vx,\vy) \cdot
    \vF_T(\vx;\vy',0) \, \dy' \, \dx\right)\, \varphi(y_3)\, \dy_3.
  \end{align*}
  The uniqueness of the limit implies that the
  whole original subsequence converges to $\alpha^0$, which completes the proof.
\end{proof}

We conclude this section with the proof of Lemma \ref{lem:alpha2}. We note
that this proof is the only instance in which the {\em special geometry} of
our framework (the fact that $\Omega = \Sigma\times(-L,L)$ where $\Sigma$
is a flat hypersurface) plays a significant role. A generalization of this
result to geometries with non-flat hypersurfaces is given in the
\ref{app:generalization}.
\begin{proof}[Proof of Lemma~\ref{lem:alpha2}]
  We prove \eqref{eq:alpbdder} by first assuming that $\vE^d$ is smooth
  enough for the all the integrals below to make sense (and conclude that
  \eqref{eq:alpbdder}  holds for our $\vE^d$ by a density argument). We
  start with the formula
  \begin{align*}
    \frac{\text{d}\alpha^d}{\text{d}t}(t) & = \sum_{k\in \Gamma^d}
    d^2\int_\Sigma \pa_{x_3} \vE^d_T(\vx',(k+t)d) \cdot
    \vF_T(\vx',(k+t)d;\vx'/d,0) \, \dx'
    \\
    &\qquad\qquad
    + \sum_{k\in \Gamma^d} d^2\int_\Sigma \vE^d_T(\vx',(k+t)d) \cdot
    [\pa_{x_3} \vF_T] (\vx',(k+t)d;\vx'/d,0) \, \dx'
    \\
    &=:\beta_1(t)+\beta_2(t).
  \end{align*}
  Now consider the second term, $\beta_2(t)$. Using the fact that
  $\pa_{x_3} \vF\in L^\infty$, we have
  \begin{align*}
    \int_0^1 |\beta_2(t)|^2\, dt & \leq \| \pa_{x_3} \vF\|_{L^\infty}^{2} d^2
    \int_0^1 \left( \sum_{k\in \Gamma^d} d \int_\Sigma \big|
    \vE^d_T(\vx',(k+t)d)\big| \, \dx'\right)^2\, \dt
    \\
    & \leq \| \pa_{x_3} \vF\|_{L^\infty}^{2}\, |\Gamma| d^2 \int_0^1 \sum_{k\in
    \Gamma^d} d \left(\int_\Sigma \big| \vE^d_T(\vx',(k+t)d)\big| \,
    \dx'\right)^2\, \dt
    \\
    & \leq \| \pa_{x_3} \vF\|_{L^\infty}^{2}\, |\Gamma|\, d^2 |\Sigma| \int_0^1
    \sum_{k\in \Gamma^d} d\int_\Sigma \big| \vE^d_T(\vx',(k+t)d)\big|^2 \,
    \dx' \, \dt
    \\
    & \leq \| \pa_{x_3} \vF\|_{L^\infty}^{2}\, |\Gamma|\, d^2 |\Sigma|
    \sum_{k\in \Gamma^d} d \int_\Sigma \int_0^1 \big|
    \vE^d_T(\vx',(k+t)d)\big|^2\, \dt \, \dx'
    \\
    & \leq \| \pa_{x_3} \vF\|_{L^\infty}^{2}\, |\Gamma|\, d^2 |\Sigma|
    \sum_{k\in \Gamma^d} \int_{kd}^{(k+1)d} \int_\Sigma \big|
    \vE^d_T(\vx',x_3)\big|^2\, \dx_3 \, \dx'
    \\
    & \leq \| \pa_{x_3} \vF\|_{L^\infty}^{2}\, |\Gamma|\, d^2 |\Sigma|
    \int_\Gamma \int_\Sigma \big| \vE^d_T(\vx',x_3)\big|^2\, \dx' \, \dx_3
    \\
    & \leq|\Omega|   d^2 \| \pa_{x_3} \vF\|_{L^\infty}^{2} \| \vE^d\|^2_{{L^2(\Omega)}}.
  \end{align*}
  To determine a bound for $\beta_1(t)$, we use the fact that the
  derivative $\pa_{x_3} \vE^d_T$ is a combination of $\na\times\vE^d$ and
  $\na_T \vE^d_3$. After expanding the dot product in $\beta_1(t)$, we end
  up with two similar terms (involving $\pa_{x_3} \vE^d_1\, \vF_1$ and
  $\pa_{x_3} \vE^d_2 \, \vF_2$ respectively), and we will find a bound for
  the first one only (the second term is handled in the same way):
  \begin{align*}
    \beta_{11}(t)& :=\sum_{k\in \Gamma^d} d^2\int_\Sigma \pa_{x_3}
    \vE^d_1(\vx',(k+t)d) \vF_1(\vx',(k+t)d,\vx'/d,0) \, \dx'
    \\
    &=
    \sum_{k\in \Gamma^d} d^2\int_\Sigma \pa_{x_3} \vE^d_1(\vx',(k+t)d)\,
    w^d_k(\vx',t) \, \dx',
  \end{align*}
  where $w^d_k(\vx',t) = \vF_1(\vx',(k+t)d,\vx'/d,0).$ Using the
  definition of the curl and integration by parts once, we can then write
  \begin{align*}
    \beta_{11}(t)& =\sum_{k\in \Gamma^d} d^2\int_\Sigma (\na\times
    \vE^d)_2(\vx',(k+t)d) w^d_k(\vx',t) \, \dx'
    \\
    &\qquad\qquad\qquad
    + \sum_{k\in \Gamma^d} d^2\int_\Sigma \pa_{x_1} \vE^d_3(\vx',(k+t)d)
    w^d_k(\vx',t) \, \dx'
    \\
    & =\sum_{k\in \Gamma^d} d^2\int_\Sigma (\na\times \vE^d)_2(\vx',(k+t)d)
    w^d_k(\vx',t) \, \dx'
    \\
    &\qquad\qquad\qquad
    - \sum_{k\in \Gamma^d} d^2\int_\Sigma \vE^d_3(\vx',(k+t)d)\pa_{x_1}
    \big[ w^d_k(\vx',t) \big]_{\Sigma} \dx'
    \\
    &\qquad\qquad\qquad\qquad\qquad\qquad
    + \sum_{k\in \Gamma^d} d^2\int_{\pa \Sigma}\vE^d_3(\vx',(k+t)d)
    w^d_k(\vx',t)\,\vn_1\, \dox'.
  \end{align*}
  By using the estimates
  \begin{align*}
    \|w^d_k(\vx',t)\|_{L^\infty}\leq \| \vF_1\|_{L^\infty},
    \quad\text{and}\quad
    \|\pa_{x_1}w^d_k(\vx',t)\|_{L^\infty}\leq \| \na_x
    \vF_1\|_{L^\infty}+\frac{1}{d}\| \na_y\vF_1\|_{L^\infty},
  \end{align*}
  we conclude that (proceeding similarly to the case with the bound for
  $\beta_2(t)$ above)
  \begin{multline*}
    \int_0^1|\beta_{11}(t)|^2\, \dt \leq C d^2 |\Gamma| |\Sigma|
    \int_\Gamma \int_\Sigma | (\na\times \vE^d)_2 (\vx',x_3)|^2\, \dx' \,
    \dx_3
    \\
    + (Cd^2+C) |\Gamma| |\Sigma| \int_\Gamma \int_\Sigma |
    \vE^d_3(\vx',x_3)|^2\, \dx' \, \dx_3
    \\
    + C d^2 |\Gamma| |\pa \Sigma| \int_\Gamma \int_{\pa \Sigma} |
    \vE^d_3(\vx',x_3)\nu_1|^2\, \dox' \, \dx_3.
  \end{multline*}
  We notice that $\vE^d_3$ is part of $\vE^d_T$ on the boundary
  $\pa\Sigma\times\Gamma\subset \pa\Omega$; thus, we have
  \begin{align*}
    \int_0^1|\beta_{11}(t)|^2\, \dt & \leq C|\Omega| d^2 \| \na\times \vE^d\|^2_{L^2(\Omega)}
    +C |\pa\Omega| d^2\|\vE^d\|^2_{L^2(\pa\Omega)}\\
    & \qquad + C |\Omega| (1+d^2) \| \vE^d\|^2_{L^2(\Omega)}.
  \end{align*}
  By combining these estimates, we obtain \eqref{eq:alpbdder}. The last
  inequality in Lemma~\ref{lem:alpha2} then follows from
  Proposition~\ref{prop:estimates}.
\end{proof}

\subsection{The cell problem and proof of Theorem \ref{thm:hom}}
\label{subsec:cell-hom}

We now turn our attention to the cell problem \eqref{eq:h3}. We will prove
the following proposition:
\begin{proposition}
  \label{prop:cell}
  Given $\vME(x)\in L^2(\Omega ;\C^3)$, the cell problem
  \eqref{eq:cellweak0} has a unique solution $\varphi(\vx,\vy) $ satisfying
  $\varphi \in L^2(\Omega;H^1_\#(Y;\C^3))$, and $(\na_y \varphi)_T \in
  L^2(\Omega\times\Sigma_0)$. Furthermore, we can write
  \begin{align}\label{eq:vphi}
    \varphi(\vx,\vy) = \sum_{j=1}^3\vec \chi_j(\vx,\vy) \vME_j(\vx),
  \end{align}
  where for a.\,e. $\vx\in\Omega$, $\vy\mapsto \vec \chi_j(\vx,\vy)$ is the
  unique solution in
  \begin{align*}
    H=\big\{u\in H^1_\# (Y;\C^3)\,;\, (\na_y u)_T\in L^2(\Sigma_0)\big\}
  \end{align*}
  of
  \begin{equation}\label{eq:cell}
    \begin{cases}
      \begin{aligned}
        & \nabla_y \cdot \big(i \omega \e(\vx,\vy) (\ve_j +\nabla_{y}
        \vec \chi_j(\vx,\vy)\big) = 0
        &\quad\text{in } Y\setminus \Sigma_0,
        \\[0.3em]
        & \big[i \omega \e(\vx,\vy) (\ve_j +\nabla_{y} \vec \chi_j(\vx,\vy)) \cdot
        \nu\big]_{\Sigma_0}
        &
        \\
        &\qquad\qquad\qquad
        \;= \nabla_T\cdot \big(\sigma(\vx,\vy) (\ve_j +\nabla_y
       \vec  \chi_j(\vx,\vy))_T\big)
       &\quad\text{on } \Sigma_0,
      \end{aligned}
    \end{cases}
  \end{equation}
  and satisfies $\chi_j \in L^\infty(\Omega;H)$.
 \end{proposition}

\begin{proof}
  We first prove the existence and uniqueness of $\vec\chi_j$ for $j=1,2,3$.
  This implies the existence of $\varphi(\vx,\vy)$ given by \eqref{eq:vphi}.
  The uniqueness of $\varphi$ can be  proved with exactly the same procedure
  as that for $\vec \chi_j$.

  The weak formulation of \eqref{eq:cell} reads
  \begin{multline}\label{eq:cellweak}
    \int_Y i \omega \e(\vx,\vy) (\ve_j +\nabla_y \chi_j(\vx,\vy)) \cdot
    \nabla_y \overline v(\vy)\,\dy\,
    \\
    = \int_{\Sigma_0} \sigma(\vx,\vy) (\ve_j +\nabla_y
    \chi_j(\vx,\vy))_T\cdot \nabla_T \overline v (\vy) \, \doy.
  \end{multline}
  We note that $\vx$ plays the role of a parameter here. Thus, for a fixed
  $\vx\in \Omega$, we  find the function $\vy\mapsto \chi_j(\vx,\vy)$ by
  solving
  \begin{equation}
    \label{eq:cell2}
    b_{\vx}(\chi_j(x,\cdot),v) = \int_Y i \omega \e(\vx,\vy) \ve_j \cdot
    \nabla_y \overline v(\vy)\,\dy- \int_{\Sigma_0} \sigma(\vx,\vy) (\ve_j)
    _T\cdot \nabla_T \overline v (\vy) \, \doy,
  \end{equation}
  with the sesquilinear form $b_x$ defined by
  \begin{align*}
    b_{\vx}(u,v): = \int_Y (-i \omega \e(\vx,\vy)) \nabla_y u(\vy) \cdot
    \nabla_y \overline v(\vy)\,\dy +\int_{\Sigma_0} \sigma(\vx,\vy)\nabla_T u
    (\vy)\cdot \nabla_T \overline v (\vy) \, \doy,
  \end{align*}
  for all functions $u(\vy)$ and $v(\vy)$ in $H$. Under the assumptions of
  Theorem \ref{thm:hom}, the form $b_x$ is continuous and coercive on $H$
  equipped with the norm
  \begin{align*}
    \| u\|_{H}^2 = \int_Y | \nabla_y u(\vy)|^2 \,\dy +\int_{\Sigma_0}
    |\nabla_T u (\vy)|^2 \, \doy.
  \end{align*}
  In particular, the coercivity follows from assumption \eqref{eq:hyp3}:
  \begin{multline}
    \label{eq:coerb}
    \Re (b_{\vx}(u,u))
    =\int_Y \Im \e (\vx,\vy)| \nabla_y u(\vy)|^2 \,\dy +\int_{\Sigma_0} \Re
    \sigma(\vx,\vy) |\nabla_T u (\vy)|^2 \, \doy
    \\
    \geq c \| u\|_{H}^2.
  \end{multline}
  The existence and uniqueness of $\chi_j$ thus follows by virtue of the
  Lax-Milgram theorem~\ref{thm:LM}. The bound in $L^\infty(\Omega;H)$
  follows from \eqref{eq:coerb}.
\end{proof}

We are finally in a
position to prove our main homogenization result:
\begin{proof}[Proof of Theorem~\ref{thm:hom}]
  Let $d_l$ be any subsequence such that $\vE^{d_l}$ and $\vH^{d_l}$
  two-scale converge as  in Lemma~\ref{lem:ts}. Then, using
  Proposition~\ref{prop:hom} we see that $\vE^{d_l}$ and $\vH^{d_l}$
  converge weakly in $L^2(\Omega;\C^3)$ to $\vME(x) \in X^0$ and
  $\vMH(x)\in L^2(\Omega;\C^3)$. Furthermore, by inserting the
  representation \eqref{eq:vphi}, which  holds true by virtue of
  Proposition~\ref{prop:cell}, into the weak formulation \eqref{eq:h1} of
  Proposition~\ref{prop:hom}, we conclude that $\vME$ solves the
  homogenized equation \eqref{eq:weaka0}.
  The uniqueness of this limit, provided by Theorem~\ref{thm:existence0},
  then implies that the whole sequences $\vE^d(\vx)$ and $\vH^d(\vx)$ are
  in fact convergent thus proving Theorem~\ref{thm:hom}. Note that the fact
  that the corrector $\vec\chi(\vx,\vy)$ solves \eqref{eq:cell-intro} is an
  immediate consequence of Proposition~\ref{prop:cell}.
\end{proof}


\section{Conclusion and discussion}
\label{sec:conclusion}

In this paper, we rigorously derived an effective description for
electromagnetic wave propagation in a plasmonic crystal consisting of
metallic sheets immersed in a non-magnetic dielectric medium. The main
result of our analysis is a formula for the macroscopic dielectric
permittivity, $\eff$, that combines a \emph{bulk average} pertaining to the
microstructure of the ambient medium and a \emph{surface average} that
takes into account the surface conductivity of each sheet. The accompanying
corrector field is subject to a cell problem in which the divergence of the
(microscale) dielectric permittivity enters as forcing along with a jump
condition across the sheet that is proportional to the surface conductivity
and involves the surface Laplacian of the corrector. In our analysis, we
made use of the well-known notion of two-scale convergence
from~\cite{allaire,nguetseng}.

It is worthwhile to compare our approach and main result to the ones
in~\cite{amirat2017}. Although that work (\cite{amirat2017}) reports a
similar result for the effective permittivity, the geometric setting (in
the context of geophysics) in~\cite{amirat2017} is different from ours. The
mathematical formulations bear a resemblance; the respective proofs,
however, are quite different. In~\cite{amirat2017}, a key tool is the
generalization of the notion of two-scale convergence to functions defined
on periodic surfaces \cite{Neuss-Radu, Allaire95}. This immediately implies
the two-scale convergence of the interfacial currents $(\sigma^d \vec
E^d_T)\delta_{\Sigma^d}$, and the difficulty is to properly identify the
corresponding limit. This part of the proof in~\cite{amirat2017} exploits
in a crucial way the particular geometry of small inclusions as opposed to
the large sheets of our work. (Note, however, that the main ideas
in~\cite{amirat2017} could certainly be adapted to our setting). The proof
that we develop in the present paper does not rely on this notion of
two-scale convergence on surfaces but instead recovers directly the
convergence of the currents to the appropriate term in the sense of
distribution; cf. Proposition \ref{prop:limsum}. This aspect of our work,
and in particular the introduction of the function $\alpha^d(t)$, is close,
in spirit at least, to the unfolding method developed in
\cite{Cioranescu06}.

From a physical viewpoint, the plasmonic structure analyzed here has been
proposed as a type of \emph{metamaterial} that may achieve the
epsilon-near-zero effect. According to this effect, a macroscopic
electromagnetic wave can propagate through the structure almost without any
phase delay. This possibility has been recently predicted for isotropic and
homogeneous metallic sheets hosted by relatively simple, anisotropic
dielectrics  (ambient media) by use of classical solutions to Maxwell's
equations via the Bloch wave theory~\cite{mattheakis2016,maier2018}. Our
analysis here is more general since it relies on intrinsic properties of
Maxwell's equations, without recourse to particular solutions. Thus, our
homogenization result is a promising tool for understanding how the
epsilon-near-zero effect can possibly emerge in a broad class of plasmonic
structures. The implications of our homogenization outcome are the subject
of work in progress.

To link our homogenization result to predictions related to the
epsilon-near-zero effect, e.g.,~\cite{mattheakis2016,maier2018}, consider
cell problem~\eqref{eq:cell-intro} in the simple case with $\nabla_y\cdot
\e(\vx,\vy)\equiv 0$. By this hypothesis, we deduce that the corrector
field must vanish, i.e., $\vec\chi(\vx,\vy)\equiv 0$. Hence,
formula~\eqref{eq:effectiveperm} for $\eff$ reduces to the average
\begin{align}
  \eff =
  \int_Y\e(\vy)\dy
  -\frac1{i\omega}
  \int_\Sigma\big\{\sigma(\vy)P_T
  (I_n)\big\}\doy,
  \label{eq:epsilon-eff-simple}
\end{align}
under the additional, simplifying assumption that the dielectric
permittivity, $\e^d$, of the ambient medium and the surface conductivity,
$\sigma^d$, of each sheet depend only on the fast coordinate of the
problem. For a plasmonic sheet such as doped graphene it is possible to
have $\Im\sigma>0$ and $\Im\sigma\gg
\Re\sigma>0$~\cite{grigorenko2012,cheng14}. Thus, by inspection
of~\eqref{eq:epsilon-eff-simple} one observes that
$\sigma=\sigma(\omega)=\sigma^d(\omega)/d$ can possibly be tuned so that
{\em at least one of the eigenvalues of $\eff$ is close to zero}. This in
turn implies that an electromagnetic wave propagating in the appropriate
direction, determined by the respective eigenvector of $\eff$, may
experience almost no phase delay. For examples in the relatively simple
setting with a diagonal $\e$ and scalar constant $\sigma$, the reader is
referred to~\cite{mattheakis2016,maier2018}.

Specifically, if one chooses $\e(\vy)={\rm
diag}(\e_x(\vy),\e_y(\vy),\e_z(\vy))$ with $\e_x={\rm const}.$,
$\e_y(\vy)=\e_z(\vy)=\e_{z,0}f(y_1)$, $\e_{z,0}={\rm const.}$ and
$\sigma={\rm const.}$ for some positive and bounded function
$f$~\cite{maier2018}, by~\eqref{eq:epsilon-eff-simple} the effective
dielectric permittivity becomes
\begin{equation*}
  \eff={\rm diag}\left(\e_x,\quad
  \e_{z,0}\int_0^1 f(y_1)\,{\rm d}y_1 +i\sigma/\omega,\quad
  \e_{z,0}\int_0^1 f(y_1)\,{\rm d} y_1+i\sigma/\omega\right).
\end{equation*}
Notice that if $\Re\sigma\approx 0$ and $\Im\sigma >0$, the two diagonal
elements of $\eff$ are close to zero if $\omega$ or $d$ is adjusted
so that the following relation holds:
\begin{align*}
  d\approx d_0:=\frac{-i\sigma^d(\omega)}{\omega\e_{z,0}}\left(\int_0^1
  f(y_1)\,{\rm d}y_1\right)^{-1}.
\end{align*}
Note that the quantity $-i\sigma^d/(\omega\e_{z,0})$ is the plasmonic
length, which expresses the scale for the decay of a surface
plasmon-polariton away from the sheet in the case of transverse-magnetic
polarization~\cite{mattheakis2016}. The condition $d\approx d_0$ has
dramatic consequences in the dispersion of macroscopic waves through the
plasmonic structure~\cite{maier2018}.

This discussion points to a few open problems with direct implications in
plasmonics. For instance, it is of interest to define the epsilon-near-zero
effect in situations where the dielectric permittivity of the ambient
medium or the conductivity of the metallic sheet also depend on slow
spatial variables (in isotropic or anisotropic settings). A related issue
is to understand the role of the corrector field if $\nabla_y\cdot
\e(\vx,\vy)\neq 0$. Our assumption that the ambient medium and sheet are
non-magnetic can be deemed as restrictive, and could in principle be
relaxed. In the presence of magnetic media, the homogenized Maxwell
equations may include an effective magnetic permeability, $\mu^{\rm eff}$,
that should combine bulk and surface averages. In fact, the jump condition
across the sheet can be generalized to also include a discontinuity in the
tangential electric field which may be relevant to the magnetoelectric
effect~\cite{Zulicke14}. This and other generalizations can lead to rich
homogenization problems in plasmonics.


\appendix

\renewcommand\thesection{\appendixname\ \Alph{section}} 
\section{Two-scale convergence: A few results}
\label{app:two-scale}
\renewcommand\thesection{\Alph{section}} 
First, we recall the following classical definition and corresponding
theorem~\cite{allaire}.
\begin{definition}
  A sequence $\vu^d$ in $L^2(\Omega;\C^3)$ is said to two-scale converge to
  $\vu^{(0)}\in L^2(\Omega\times Y ;\C^3)$ if
  \begin{equation*}
    \lim_{d\to 0 } \int_\Omega \vu^d (\vx) \cdot \vPsi(\vx,\vx/d)\, \dx
    =
    \int_\Omega\int_Y \vu^{(0)}(\vx,\vy) \cdot \vPsi(\vx,\vy)\, \dy\,\dx,
  \end{equation*}
  for all test functions $\vPsi \in C_0 (\Omega; C_\#(Y;\C^3))$.
\end{definition}
\begin{theorem}
  \label{thm:ts}
  If the sequence $\vu^d$ is bounded in $L^2(\Omega;\C^3)$, then there
  exists a subsequence which two-scale converge to a function
  $u^{(0)}(\vx,\vy)$. Furthermore, the sequence $\vu^d$ weakly converges in
  $L^2(\Omega;\C^3)$ to the function
  \begin{align*}
    \bar u(\vx) =\int_{Y}u^{(0)}(\vx,\vy)\, dy.
  \end{align*}
\end{theorem}
Next, we prove the following lemma which is also relevant to our exposition.
\begin{lemma}
  \label{lem:homf}
  Let $\vf \in L^2_\#(Y;\C^3)$ be such that
  \begin{align*}
    \nabla_y \times \vf (\vx,\vy) = 0
    \quad\text{in } \mathcal D'(\Omega\times Y),\qquad \int_Y
    \vf(\vx,\vy)\, \dy =0.
  \end{align*}
  Then, there exists a scalar function $\varphi(\vx,\vy)\in
  L^2(\Omega;H^1_\#(Y))$ such that
  \begin{align*}
    f(\vx,\vy) = \nabla_y\varphi(\vx,\vy).
  \end{align*}
\end{lemma}
This result is a slight variation of Lemma B.5 in
\cite{wellander2003}. We give the proof for the sake of completeness.
\begin{proof}
  We can write the following Fourier expansion of $\vf$ in $Y$:
  \begin{align*}
    \vf(\vx,\vy) = \sum_{\vk\in\Z^3} \vc_{\vk}(\vx) e^{i2\pi \vk\cdot \vy}.
  \end{align*}
  The conditions on $\vf$ imply that $\vc_\vk\times \vk =0$ for all
  $\vk\in\Z^3$, $\vc_0=0$. In particular, for $\vk\neq 0$, the vector
  $\vc_\vk$ is parallel to $\vk$; and if we define $d_\vk=\frac{\vc_\vk\cdot
  \vk}{i2\pi|\vk|^2}$ then we have $\vc_\vk = (i2\pi) d_l \vk$. This
  in turn implies that the function
  \begin{align*}
    \varphi(\vx,\vy) = \sum_{\vk\in\Z^3\setminus\{0\}} d_l e^{i2\pi
    \vk\cdot \vy}
  \end{align*}
  satisfies
  \begin{align*}
    \na _y \varphi(\vx,\vy) = \sum_{\vk\in\Z^3\setminus\{0\}} i2\pi d_\vk
    \vk e^{i2\pi \vk\cdot \vy} = \vf(\vx,\vy).
  \end{align*}
  Furthermore, $\varphi(\vx,\vy)\in L^2(\Omega;H^1_\#(Y))$ since $\vf\in
  L^2_\#(Y;\C^3)$.
\end{proof}


\renewcommand\thesection{\appendixname\ \Alph{section}} 

\section{General hypersurfaces $\Sigma^d$}

\renewcommand\thesection{\Alph{section}} 
\label{app:generalization}
In this section, we generalize the main result to non-flat hypersurfaces
$\Sigma$. In particular, we show how to prove our main
result~\eqref{eq:effectiveperm}
when the hypersurface $\Sigma$ is not necessarily a plane, but forms the
graph of a smooth $Y'$-periodic function $h$ (for which, for simplicity, we
assume that $-1\leq h\leq 1$). More precisely, we still assume that the
domain $\Omega$ has the form
\begin{align*}
  \Omega = \Sigma' \times \Gamma,
\end{align*}
where $\Sigma'$ is a smooth bounded subset of $\R^2$ and $\Gamma = (-L,L)$.
But we now take
\begin{align*}
  \Sigma^d = \cup_{k\in \Gamma^d} \{ (\vx',dh(\vx'/d)+kd)\; ;\;
  \vx'\in\Sigma'\},
\end{align*}
where $\Gamma^d = \{k\in\Z\, ;\, kd\in(-L+d,L-d)\}$. Note that this
definition of $\Gamma^d$ (and the assumption $-1\leq h\leq 1$) ensures that
$\Sigma^d$ does not intersect  the boundaries $\Sigma'\times\{-L\}$ and
$\Sigma'\times\{L\}$. Finally, we recall that $\Sigma_0$ denotes the graph
of $h$ in $Y$:
\begin{align*}
  \Sigma_0 = \{ (y',h(y)) \, ;\, y'\in Y ' \}.
\end{align*}
%
The only part in the proof of Theorem \ref{thm:hom} that utilized the
particular structure of $\Sigma^d$ was in the proof of Proposition
\ref{prop:limsum}. We will thus show in the following that the result of
Proposition \ref{prop:limsum} still holds in the more general framework
described above. In order to state the corresponding result, we introduce
the matrix $P(\vy')$, which expresses the projection onto the tangent space
of $\Sigma_0$ at the point $(\vy',h(\vy'))$. For $\vx \in \Sigma^d$, we
thus have
\begin{align*}
  \vE^d_T(\vx) = P(\vx'/d) \vE^d(\vx).
\end{align*}
Our goal is then to prove the following proposition.
\begin{proposition}\label{prop:limsumgen}
  Assume that $h\in W_\#^{2,\infty}(Y')$ and recall that $\vE^d$ is bounded
  in $X^d$ and two-scale converges to the function $ \vE^{(0)}(\vx,\vy)$.
  Then, for all functions $\vF(\vx,\vy)$ defined in $\Omega\times Y$ that
  are periodic with respect to $y$ and admit $\vF,\; \na_x \vF,\; \na_y \vF
  \in L^\infty(\Omega\times Y)$, we have
  \begin{multline}
    \lim_{d\to 0} d \int_{\Sigma^d}   \vE^d_T(\vx) \cdot \vF_T(\vx,\vx/d)\,
    \dox
    =
    \\
    \int_\Omega\int_{\Sigma_0} P(\vy) \vE^{(0)}(\vx,\vy)\cdot
    P(\vy)F(\vx,\vy)\, \doy \dx.
  \end{multline}
\end{proposition}

\begin{proof}
  As in the proof of Proposition \ref{prop:limsum}, the key step is the
  introduction of the following function (defined for $t\in(0,1)$):
  \begin{align}
    \alpha^d(t)
    & =
    d\int_{\Sigma^d} \vE^d_T(\vx',x_3+td) \cdot
    \vF_T(\vx',x_3+td;\vx'/d,x_3/d) \, \dox \label{eq:alphaapp} \\
    & =
    \sum_{k\in \Gamma^d} d
    \int_{\Sigma'}
    P\left(\tfrac{\vx'}{d}\right)\vE^d\left(\vx',(k+t)d +dh(\vx'/d) \right)
    \nonumber \\
    &\qquad
    \cdot P\left(\tfrac{\vx'}{d}\right)\vF \left(\vx',(k+t)d +dh(\vx'/d);\vx'/d,h(\vx'/d)\right)
    \sqrt{1+\left|\na h\left(\tfrac{\vx'}{d}\right)\right|^2} \,  d\vx'.  \nonumber
  \end{align}
  The main difficulty is to derive the appropriate bounds on $\alpha^d$ and
  its derivative (see Lemma   \ref{lem:alpha2}). For this purpose, we
  introduce the diffeomorphisms $\vg: \R^3\to \R^3$ and  $\vg^d: \R^3\to
  \R^3$ defined by $\vg(\vx) := (\vx',h(\vx')+x_3),$ and $\vg^d(\vx) := d
  \vg(\vx/d) = (\vx',dh(\vx'/d)+x_3)$. We have $\Sigma^d =
  \vg^d(\wSigma^d)$, where
  \begin{align*}
    \wSigma^d=\cup_{k\in \Gamma^d}\Sigma'\times \{kd\},
  \end{align*}
  and $\Omega = \vg^d(\wOmega^d)$, where
  \begin{align*}
    \wOmega^d =
    \{(x',x_3)\, ;\, x'\in\Sigma', \; -L-d h(x'/d)\leq x_3 \leq L-d
    h(x'/d)\}
  \end{align*}
  (note that $|\Omega \Delta \wOmega^d|\leq Cd$).
  We also define
  \begin{align*}
    \wE^d (\vx) = \vE^d\big(\vg^d(\vx)\big)\nabla\vg^d(\vx), \qquad x\in
    \wOmega^d
  \end{align*}
  that is, $\wE^d_i(\vx) =
  \sum_{j=1}^3\vE_j\big(\vg^d(\vx)\big)\pa_i\vg^d_j(\vx)$. This is a
  natural definition when $\vE^d$ is the gradient of a potential (that is
  when $\vE^d$ is curl free). We will see below that this change of
  function also preserves the curl estimates that played a crucial role in
  the proof of Lemma \ref{lem:alpha2}.
  More precisely, we will make use of the following properties:
  \begin{enumerate}
    \item
      Since $\nabla\vg^d(\vx)= \nabla\vg(\vx/d)$, we have $\|\pa_i \vg^d_j
      \|_{L^\infty}\leq C$ for all $i$, $j$, independently of $d$.
      Furthermore, a simple computation gives
      \begin{equation}
        \label{eq:jac}
        |\det \nabla\vg^d(\vx)|=1.
      \end{equation}
      In particular, we have (with $d$-independent constants)
      \begin{equation}
        \label{eq:wEL2}
        \int_{\wOmega^d} |\wE^d(\vx)|^2 \, d\vx \leq  C \int_{\Omega}
        |\vE^d(\vx)|^2 \, d\vx \leq C.
      \end{equation}
    \item
      For $\vx\in \wSigma^d$, the projection $\wE^d_T$ onto the tangent
      plane to $\wSigma^d$ only depends on $\vE^d_T$, the projection of
      $\vE^d$ onto the tangent plane to $\Sigma^d$. Indeed we can write
      \begin{align*}
        \wE^d_T(\vx)
        &= (\wE^d_1(\vx),\wE^d_2(\vx),0)^T
        \\
        &= (\pa_1 \vg^d(\vx)\cdot \vE^d\big(\vg^d(\vx)\big),\pa_2
        g^d(\vx)\cdot \vE^d\big(\vg^d(\vx)\big),0 )^T
        \\
        & = (\pa_1 \vg^d(\vx)\cdot \vE^d_T\big(\vg^d(\vx)\big),\pa_2
        \vg^d(\vx)\cdot \vE^d_T\big(\vg^d(\vx)\big),0 )^T.
      \end{align*}
      In the last equality we used the fact that $\pa_1 g^d$ and $\pa_2
      g^d$ are tangent vectors to $\Sigma^d$. Using the fact that the
      vector $(-\pa_1h,-\pa_2h,1)$ is normal to $\Sigma_d$ (it is the
      vector $\pa_1 g^d\times \pa_2 g^d$), we can rewrite this equality as
      \begin{equation}\label{eq:ET}
        \wE^d_T(\vx) = M (\vx'/d) \vE^d_T\big(\vg^d(\vx)\big),
      \end{equation}
      with the matrix
        \begin{align*}
          M(\vx') =
          \begin{pmatrix}
            1 & 0 & \pa_1 h(\vx') \\
            0 & 1 & \pa_2 h(\vx') \\
            -\pa_1 h(\vx') & -\pa_2 h(\vx') & 1
          \end{pmatrix}.
      \end{align*}
      This $M(\vx')$ is smooth and invertible. (The latter attribute can be readily
      deduced from the determinant of $M(\vx')$, which is
      $1+|\pa_1h(\vx'/d)|^2+|\pa_2h(\vx'/d)|^2$). We can also write
      \begin{equation}
        \label{eq:ET2}
        \wE^d_T(\vx) = M (\vx'/d) P(\vx'/d)\vE^d\big(\vg^d(\vx)\big).
      \end{equation}
    \item
      The definition of $g$ immediately gives $\wE^d_3(\vx)= \vE^d_3
      \big(\vg^d(\vx)\big)$ in $\wOmega^d$. Furthermore, since $\vE^d_3$ is
      part of $\vE^d_T$ on $\pa \wSigma\times\Gamma$, using
      \eqref{eq:jac} we conclude that
      \begin{equation}
        \label{eq:E3}
        \int_{\Gamma^{d}} \int_{\Sigma'} |\wE^d_3(\vx',x_3)|^2\,
        d\vx'\, \dx_3 \leq \int_{\Gamma^{d}} \int_{\Sigma}
        |\vE^d_3(\vx',x_3)|^2\, d\vx'\, \dx_3 \leq C.
      \end{equation}
    \item
      The curl of $\wE^d$ only depends on the components of $\na\times
      \vE^d$, and is thus bounded in $L^2$. Indeed, writing $\wE^d (\vx)=
      \sum_{j=1}^3 E_j^d(\vg^d(\vx)) \na \vg^d_j (\vx) $, we find
      \begin{align*}
        \na\times \wE^d (\vx)
        & = \sum_{j=1}^3 \na (\vE_j^d\big(\vg^d(x)\big) \times \na \vg^d_j
        (\vx)
        \\
        &= \sum_{j,l=1}^3 \pa_l \vE_j^d\big(\vg^d(x)\big)\na \vg^d_l(x)
        \times \na \vg^d_j (\vx),
      \end{align*}
      where we used the chain rule $\na (\vE_j (\vg (\vx)) = \sum_{l=1}^3
      \pa_l \vE_j (\vg(\vx)) \na \vg_l(\vx)$. Using the anti-symmetry of
      the cross product, we deduce the relation 
      \begin{align*}
        \na\times \wE^d (\vx) =\frac 1 2 \sum_{j\neq l } \na \vg^d_l(\vx)
        \times \na \vg^d_j (\vx) \Big[\pa_l \vE_j^d (\vg^d(\vx))-\pa_j
        \vE_l^d (\vg^d(\vx))\Big].
      \end{align*}
      We thus see that $\na\times \wE^d (\vx) $ depends in a linear fashion on
      the components of $\na\times \vE^d \big(\vg^d(\vx)\big)$ and, since
      $\|\pa_i \vg^d_j \|_{L^\infty}\leq C$, we obtain 
      \begin{align*}
        | \na\times \wE^d (\vx) |^2 \leq C |\na\times \vE^d (\vg^d(\vx))|^2.
      \end{align*}
      In particular, \eqref{eq:jac} implies the estimates 
      \begin{equation}
        \label{eq:curl}
        \int_{\wOmega^d} | \na\times \wE^d (\vx) |^2  \, d\vx \leq  C
        \int_\Omega  |\na\times \vE^d (\vx)|^2\, d\vx
        \leq C.
      \end{equation}
  \end{enumerate}

  We are now ready to prove Proposition  \ref{prop:limsumgen}. Using the
  change of variable introduced above, we can rewrite the function
  $\alpha^d(t)$ in  a form similar to the one  appearing in Proposition
  \ref{prop:limsum}. Indeed, using \eqref{eq:alphaapp} and \eqref{eq:ET},  we
  find that 
  \begin{align*}
    \alpha^d(t)
    =
    \sum_{k\in \Gamma^d} d\int_{\Sigma'} \wE^d_T(\vx',(k+t)d) \cdot
    \wF^d_T(\vx',(k+t)d;\vx'/d,0) \, \dx',
    \end{align*}
  where
  \begin{align*}
  \wF_T^d(\vx,\vy)
    &=
    (M(\vy')^{-1})^T \vF_T(\vg^d(\vx),\vg(\vy))\sqrt{1+|\na h(\vy')|^2}
    \\
    &=
    (M(\vy')^{-1})^T \vF_T(\vx',d h(\vy')+x_3,\vg(\vy))\,\sqrt{1+|\na
    h(\vy')|^2}.
  \end{align*}
  In order to finalize this step, we only need to show that the results of
  Lemma \ref{lem:alpha1} and Lemma \ref{lem:alpha2} hold in our framework.
  The proof of Lemma \ref{lem:alpha2} requires only appropriate bounds on
  $\wE^d$ and $\wF^d$. In particular, we realize that $\wF^d$ satisfies
  \begin{align*}
    \|\wF^d\|_{L^\infty(\Omega\times Y)} \leq C, \quad \|\na_x
    \wF^d\|_{L^\infty(\Omega\times Y)} \leq C, \quad
    \|\na_y\wF^d\|_{L^\infty(\Omega\times Y)} \leq C,
  \end{align*}
  (with constant $C$ independent of $d$) which, together with the bounds
  \eqref{eq:wEL2}, \eqref{eq:E3} and \eqref{eq:curl} are all that we need
  to prove Lemma \ref{lem:alpha2}. These same bounds are also sufficient to
  show that $\alpha^d(t)$ is bounded in $L^2(0,1)$. To prove Lemma
  \ref{lem:alpha1}, we therefore only need to identify the limit of
  $\int_0^1 \alpha^d(t)\varphi(t)\dt$. Using \eqref{eq:ET2} and
  \eqref{eq:alphaapp}, we write
  \begin{align*}
    \int_0^1 \alpha^d(t)\varphi(t)\dt
    &=
    \int_{\Sigma\times \tilde \Gamma^d}M \left(\tfrac{\vx'}{d}\right)
    P\left(\tfrac{\vx'}{d}\right)\vE^d\left(\vx',d
    h\left(\tfrac{\vx'}{d}\right)+x_3\right)
    \\
    &\qquad
    \cdot\Big(M\left(\tfrac{\vx'}{d}\right)^{-1}\Big)^T \vF_T
    \left(\vx',d h\left(\tfrac{\vx'}{d}\right)
    +x_3,\vg\left(\tfrac{\vx'}{d},0\right)\right) \varphi(x_3/d)\,
    \\
    &\qquad\qquad \qquad\qquad \qquad\qquad \qquad\qquad \qquad\qquad
    \sqrt{1+\left|\na h\left(\tfrac{\vx'}{d}\right)\right|^2} d\vx
    \\
    &=
    \int_{\Omega'} P\left(\tfrac{\vx'}{d}\right)\vE^d\left(\vx',x_3\right)
    \cdot P\left(\tfrac{\vx'}{d}\right)\vF \left(\vx',x_3,\vg
    \left(\tfrac{\vx'}{d},0\right)\right)
    \\
    &\qquad\qquad \qquad\qquad \qquad\qquad
    \varphi\left(\tfrac{x_3}{d}-h\left(\tfrac{\vx'}{d}\right)\right)
    \sqrt{1+\left|\na h\left(\tfrac{\vx'}{d}\right)\right|^2} \,d\vx
  \end{align*}
  with $\tilde\Gamma ^d= \cup_{k\in\Gamma^d} [kd,(k+1)d] = [d k_0,d k_1]$
  and $\Omega' = \{(\vx',x_3)\,;\,\vx'\in\Sigma',\;  d k_0 + d h(\vx'/d)
  <x_3<d k_1 + d h(\vx'/d)\}$. The usual properties of two-scale
  convergence imply that the above expression converges to
  \begin{multline*}
    \int_{\Omega}\int_Y P(\vy')\vE^{(0)}(\vx,\vy) \cdot
    P(\vy')\vF \left(\vx,\vg\left(\vy',0\right)\right)
    \varphi(y_3-h\left(\vy'\right)) \sqrt{1+\left|\na
    h\left(\vy'\right)\right|^2} \,d\vy\, d\vx
    \\
    = \int_{\Omega}\int_Y P(\vy')\vE^{(0)}(\vx,\vy)\cdot P(\vy')\vF
    \left(\vx,\vy',h(\vy')\right)
    \\
    \varphi(y_3-h\left(\vy'\right))
    \sqrt{1+\left|\na h\left(\vy'\right)\right|^2} \,d\vy\, d\vx
    \\
    = \int_{\Omega}\int_Y P(\vy')\vE^{(0)}(\vx,\vy',y_3+h(\vy'))\cdot
    P(\vy')\vF \left(\vx,\vy',h(\vy')\right)
    \\
    \varphi(y_3) \sqrt{1+\left|\na h\left(\vy'\right)\right|^2}
    \,d\vy\,d\vx.
  \end{multline*}
  It follows that $\alpha^d(t)$ converges $L^2(0,1)$ weakly to
  \begin{align*}
    \alpha^0(t)  &=  \int_{\Omega}\int_Y P(\vy')\vE^{(0)}(\vx,\vy',t+h(\vy'))\cdot
    P(\vy')\vF \left(\vx,\vy',h(\vy')\right)\notag\\
    &\qquad \times \sqrt{1+\left|\na h\left(\vy'\right)\right|^2}
    \,d\vy\,d\vx.
  \end{align*}
  Analogously to the proof of Proposition \ref{prop:limsum}, we can
  thus conclude that
  \begin{align*}
    &
    \lim_{d\to 0} d \int_{\Sigma^d} \vE^d_T(\vx) \cdot \vF_T(\vx,\vx/d)\, \dox
    \\
    & = \alpha^0(0)\\
    & = \int_{\Omega}\int_{Y'} P(\vy')\vE^{(0)}(\vx,\vy',h(\vy')) \cdot
    P(\vy')\vF \left(\vx,\vy',h(\vy')\right)
    \sqrt{1+\left|\na h\left(\vy'\right)\right|^2} \,d\vy'\, d\vx
    \\
    &=
    \int_{\Omega}\int_{\Sigma_0} P(\vy)\vE^{(0)}(\vx,\vy)\cdot P(\vy)\vF
    \left(\vx,\vy\right)\,\doy\, d\vx.
  \end{align*}
  This completes the proof of Proposition \ref{prop:limsumgen}.
\end{proof}



\section*{Acknowledgments}
We wish to thank Professors Mitchell Luskin and Robert V. Kohn for useful
discussions. The first two authors (MM and DM) have been supported by ARO
MURI via Award No. W911NF-14-1-0247. The authors acknowledge partial
support by NSF through Awards DMS-1912847 (MM), DMS-1412769 (DM), and
DMS-1501067 (AM).


\bibliographystyle{abbrvnat}

\end{document}